\documentclass[12p]{amsart}
\usepackage{amssymb}
\usepackage{amsmath}
\usepackage{amsfonts}
\usepackage{geometry}
\usepackage{graphicx}
\usepackage{mathrsfs,amssymb,latexsym}

\usepackage{hyperref}
\usepackage{cleveref}

\theoremstyle{plain}

\newtheorem{corollary}{Corollary}

\newtheorem{definition}{Definition}

\newtheorem{lemma}{Lemma}

\newtheorem{remark}{Remark}

\newtheorem{theorem}{Theorem}
\numberwithin{equation}{section}

\begin{document}
\title{Global well-posedness of the energy subcritical nonlinear wave equation with initial data in a critical space}
\date{\today}
\author{Benjamin Dodson}
\maketitle

\begin{abstract}
In this paper we prove global well-posedness for the defocusing, energy-subcritical, nonlinear wave equation on $\mathbb{R}^{1 + 3}$ with initial data in a critical Besov space. No radial symmetry assumption is needed.
\end{abstract}

\section{Introduction}
In this paper, we continue the study of the defocusing, cubic nonlinear wave equation,
\begin{equation}\label{1.1}
u_{tt} - \Delta u + |u|^{p - 1} u = 0, \qquad u(0,x) = u_{0}, \qquad u_{t}(0,x) = u_{1}, \qquad 3 < p < 5,
\end{equation}
with initial data in a critical space. A critical space is a space that is invariant under the scaling symmetry. Observe that $(\ref{1.1})$ is invariant under the scaling symmetry
\begin{equation}\label{1.2}
u(t,x) \mapsto \lambda^{\frac{2}{p - 1}} u(\lambda t, \lambda x), \qquad \lambda > 0.
\end{equation}
Under the above scaling symmetry, the size of the initial data changes by a factor of
\begin{equation}\label{1.3}
\| \lambda u_{0}(\lambda x) \|_{\dot{H}^{s}(\mathbb{R}^{3})} = \lambda^{s + \frac{2}{p - 1} - \frac{3}{2}} \| u_{0} \|_{\dot{H}^{s}(\mathbb{R}^{3})}, \qquad \| \lambda^{2} u_{1}(\lambda x) \|_{\dot{H}^{s - 1}(\mathbb{R}^{3})} = \lambda^{s + \frac{2}{p - 1} - \frac{3}{2}} \| u_{1} \|_{\dot{H}^{s - 1}(\mathbb{R}^{3})}.
\end{equation}
Thus, $(\ref{1.1})$ is called $\dot{H}^{s_{c}} \times \dot{H}^{s_{c} - 1}$-critical when
\begin{equation}\label{1.3.1}
s_{c} = \frac{3}{2} - \frac{2}{p - 1},
\end{equation}
 because this norm is invariant under $(\ref{1.2})$.

The scaling symmetry $(\ref{1.2})$ completely determines the local well-posedness theory for $(\ref{1.1})$.

\begin{theorem}\label{t1.1}
Equation $(\ref{1.1})$ is locally well-posed for initial data in $(u_{0}, u_{1}) \in \dot{H}^{s_{c}}(\mathbb{R}^{3}) \times \dot{H}^{s_{c} - 1}(\mathbb{R}^{3})$ on some interval $[-T(u_{0}, u_{1}), T(u_{0}, u_{1})]$. The time of well-posedness $T(u_{0}, u_{1})$ depends on the profile of the initial data $(u_{0}, u_{1})$, not just its size. For data sufficiently small in $\dot{H}^{s_{c}} \times \dot{H}^{s_{c} - 1}$, global well-posedness and scattering hold.

Additional regularity is enough to give a lower bound on the time of well-posedness. Therefore, there exists some $T(\| u_{0} \|_{\dot{H}^{s}}, \| u_{1} \|_{\dot{H}^{s - 1}}) > 0$ for any $s_{c} < s < \frac{3}{2}$.

Negatively, equation $(\ref{1.1})$ is ill-posed for $u_{0} \in \dot{H}^{s}(\mathbb{R}^{3})$ and $u_{1} \in \dot{H}^{s - 1}(\mathbb{R}^{3})$ when $s < s_{c}$.
\end{theorem}
\begin{proof}
See \cite{lindblad1995existence}.
\end{proof}

Local well-posedness is defined in the usual way.
\begin{definition}[Locally well-posed]\label{d1.2}
The initial value problem $(\ref{1.1})$ is said to be locally well-posed if there exists an open interval $I \subset \mathbb{R}$ containing $0$ such that:
\begin{enumerate}
\item A unique solution $u \in L_{t}^{\infty} \dot{H}^{s_{c}}(I \times \mathbb{R}^{3}) \cap L_{t,loc}^{2(p - 1)} L_{x}^{2(p - 1)}(I \times \mathbb{R}^{3})$, $u_{t} \in L_{t}^{\infty} \dot{H}^{s_{c} - 1}(I \times \mathbb{R}^{3})$ exists.

\item The solution $u$ is continuous in time, $u \in C(I ; \dot{H}^{s_{c}}(\mathbb{R}^{3}))$, $u_{t} \in C(I ; \dot{H}^{s_{c} - 1}(\mathbb{R}^{3}))$.

\item The solution $u$ depends continuously on the initial data in the topology of item one.
\end{enumerate}
\end{definition}

The scaling symmetry $(\ref{1.2})$ also completely determines the long time behavior of $(\ref{1.1})$ with radially symmetric initial data.
\begin{theorem}\label{t1.3}
For $3 \leq p < 5$, the initial value problem $(\ref{1.1})$ is globally well-posed and scattering for radial initial data $(u_{0}, u_{1}) \in \dot{H}^{s_{c}}(\mathbb{R}^{3}) \times \dot{H}^{s_{c} - 1}(\mathbb{R}^{3})$. Moreover, there exists a function $f : [3, 5) \times [0, \infty) \rightarrow [0, \infty)$ such that if $u$ solves $(\ref{1.1})$ with initial data $(u_{0}, u_{1}) \in \dot{H}^{s_{c}} \times \dot{H}^{s_{c} - 1}$, then
\begin{equation}\label{1.4}
\| u \|_{L_{t,x}^{2(p - 1)}(\mathbb{R} \times \mathbb{R}^{3})} \leq f(p, \| u_{0} \|_{\dot{H}^{s_{c}}(\mathbb{R}^{3})} + \| u_{1} \|_{\dot{H}^{s_{c} - 1}(\mathbb{R}^{3})}).
\end{equation}
\end{theorem}
\begin{proof}
This was proved in \cite{dodson2018global} when $p = 3$ and in \cite{dodson2018global2} when $3 < p < 5$.
\end{proof}

\begin{remark}
The argument in \cite{lindblad1995existence} may be used to show that $(\ref{1.4})$ is equivalent to scattering in the critical Sobolev norm.
\end{remark}
\begin{definition}[Scattering]\label{d1.4}
A solution to $(\ref{1.1})$ with initial data $(u_{0}, u_{1})$ is said to be scattering in some $\dot{H}^{s}(\mathbb{R}^{3}) \times \dot{H}^{s - 1}(\mathbb{R}^{3})$ if there exist $(u_{0}^{+}, u_{1}^{+}), (u_{0}^{-}, u_{1}^{-}) \in \dot{H}^{s} \times \dot{H}^{s - 1}$ such that
\begin{equation}\label{1.5}
\lim_{t \rightarrow +\infty} \| (u(t), u_{t}(t)) - S(t)(u_{0}^{+}, u_{1}^{+}) \|_{\dot{H}^{s} \times \dot{H}^{s - 1}} = 0,
\end{equation}
and
\begin{equation}\label{1.6}
\lim_{t \rightarrow -\infty} \| (u(t), u_{t}(t)) - S(t)(u_{0}^{-}, u_{1}^{-}) \|_{\dot{H}^{s} \times \dot{H}^{s - 1}} = 0,
\end{equation}
where $u$ is the solution to $(\ref{1.1})$ with initial data $(u_{0}, u_{1})$ and $S(t)(f, g)$ is the solution operator to the linear wave equation. That is, if $(u(t), u_{t}(t)) = S(t)(f, g)$, then
\begin{equation}\label{1.7}
u_{tt} - \Delta u = 0, \qquad u(0,x) = f, \qquad u_{t}(0,x) = g.
\end{equation}

Equation $(\ref{1.1})$ is called scattering for data in a certain subset $X$ if the solution to $(\ref{1.1})$ with initial data in $X$ is globally well-posed, the solution scatters both forward and backward in time, and the scattering states $(u_{0}^{+}, u_{1}^{+})$ and $(u_{0}^{-}, u_{1}^{-})$ depend continuously on the initial data.
\end{definition}

An important stepping stone in the proof of Theorem $\ref{t1.3}$ was the result of \cite{dodson2018globalAPDE} for radially symmetric initial data in a critical Besov space.
\begin{theorem}\label{t1.5}
The defocusing, cubic nonlinear wave equation $((\ref{1.1})$ when $p = 3)$ is globally well-posed and scattering for radially symmetric initial data $u_{0} \in B_{1,1}^{2}$ and $u_{1} \in B_{1,1}^{1}$. $B_{p, q}^{s}$ is the Besov space defined by the norm
\begin{equation}\label{1.8}
\| u \|_{B_{p, q}^{s}(\mathbb{R}^{3})} = (\sum_{j} 2^{jsp} \| P_{j} u \|_{L^{q}}^{p})^{1/p}.
\end{equation}
The operator $P_{j}$ is the usual Littlewood--Paley projection operator.
\end{theorem}

In this paper we generalize Theorem $\ref{t1.5}$ to the case when $3 < p < 5$ with nonradial initial data.
\begin{theorem}\label{t1.6}
Equation $(\ref{1.1})$ is globally well-posed when $3 < p < 5$ initial data $u_{0} \in B_{1,1}^{s_{c} + \frac{3}{2}}$ and $u_{1} \in B_{1,1}^{s_{c} + \frac{1}{2}}$. Furthermore, there exists $f : [3, 5) \times [0, \infty) \rightarrow [0, \infty)$ such that
\begin{equation}\label{1.8.1}
\| u \|_{L_{t,x}^{2(p - 1)}(\mathbb{R} \times \mathbb{R}^{3})} \leq f(p, \| u_{0} \|_{B_{1,1}^{s_{c} + \frac{3}{2}}(\mathbb{R}^{3})} + \| u_{1} \|_{B_{1,1}^{s_{c} + \frac{1}{2}}(\mathbb{R}^{3})}).
\end{equation}
\end{theorem}

The $B_{1,1}^{s_{c} + \frac{3}{2}} \times B_{1,1}^{s_{c} + \frac{1}{2}}$ norm is invariant under the scaling symmetry $(\ref{1.2})$. By the Sobolev embedding theorem, $B_{1,1}^{s_{c} + \frac{3}{2}} \subset \dot{H}^{s_{c}}$ and $B_{1,1}^{s_{c} + \frac{1}{2}} \subset \dot{H}^{s_{c} - 1}$. The main advantage that $B_{1,1}^{s_{c} + \frac{3}{2}} \times B_{1,1}^{s_{c} + \frac{1}{2}}$ provides is the dispersive estimate for the wave equation
\begin{equation}\label{1.9}
\| S(t)(u_{0}, u_{1}) \|_{L^{\infty}} \lesssim \frac{1}{t} \| (\nabla^{2} u_{0}, \nabla u_{1}) \|_{L^{1} \times L^{1}},
\end{equation}
which implies good behavior for the solution to the linear wave equation with initial data $(u_{0}, u_{1})$ for $t \neq 0$. Therefore, a helpful heuristic in thinking about Theorem $\ref{t1.5}$ is that blowup of a solution to $(\ref{1.1})$ with initial data in $B_{1,1}^{2} \times B_{1,1}^{1}$ must occur when $t = 0$ if it occurs at all. Radial symmetry further implies that the blowup must occur at the origin in space and time.

The results in Theorem $\ref{t1.3}$ addressed initial data that was merely radially symmetric, but not in $B_{1,1}^{2} \times B_{1,1}^{1}$, so the blowup could occur at any time, but only at the origin in space, $x = 0$. Theorem $\ref{t1.6}$ approaches this problem from the other direction. The fact that $(u_{0}, u_{1}) \in B_{1,1}^{s_{c} + \frac{3}{2}} \times B_{1,1}^{s_{c} + \frac{1}{2}}$ means that, heuristically, the blowup may occur anywhere in $\mathbb{R}^{3}$, but only at time $t = 0$, if it occurs at all.

\subsection{Outline of the proof}
The only obstacle to proving Theorem $\ref{t1.6}$ is that the $\dot{H}^{s_{c}} \times \dot{H}^{s_{c} - 1}$ norm of $(u, u_{t})$ may blow up either forward and backward in time.
\begin{theorem}\label{t1.7}
Suppose $(u_{0}, u_{1}) \in \dot{H}^{s_{c}}(\mathbb{R}^{3}) \times \dot{H}^{s_{c} - 1}(\mathbb{R}^{3})$ and $u$ solves $(\ref{1.1})$ on a maximal interval $0 \in I \subset \mathbb{R}$, with $3 < p < 5$ and
\begin{equation}\label{1.9.1}
\sup_{t \in I} \| u(t) \|_{\dot{H}^{s_{c}}(\mathbb{R}^{3})} + \| u_{t}(t) \|_{\dot{H}^{s_{c} - 1}(\mathbb{R}^{3})} < \infty.
\end{equation}
Then $I = \mathbb{R}$ and the solution $u$ scatters both forward and backward in time.
\end{theorem}
\begin{proof}
This theorem was proved in \cite{dodson2020scattering}.
\end{proof}

While there is no known conserved quantity that controls the $\dot{H}^{s_{c}} \times \dot{H}^{s_{c} - 1}$ norm of $(u(t), u_{t}(t))$ for a solution to $(\ref{1.1})$ with generic initial data $(u_{0}, u_{1}) \in \dot{H}^{s_{c}} \times \dot{H}^{s_{c} - 1}$, a solution to $(\ref{1.1})$ does have the conserved energy
\begin{equation}\label{1.10}
E(u(t)) = \frac{1}{2} \int |\nabla u(t,x)|^{2} dx + \frac{1}{2} \int u_{t}(t,x)^{2} dx + \frac{1}{p + 1} \int |u(t,x)|^{p + 1} dx = E(u(0)).
\end{equation}
For $u_{0} \in \dot{H}^{1} \cap \dot{H}^{s_{c}}$ and $u_{1} \in L^{2} \cap \dot{H}^{s_{c} - 1}$, the Sobolev embedding theorem implies
\begin{equation}\label{1.11}
\| u(0) \|_{L^{p + 1}(\mathbb{R}^{3})}^{p + 1} \lesssim \| u_{0} \|_{\dot{H}^{s_{c}}(\mathbb{R}^{3})}^{p - 1} \| u_{0} \|_{\dot{H}^{1}(\mathbb{R}^{3})}^{2},
\end{equation}
so
\begin{equation}\label{1.12}
E(u(0)) \lesssim_{\| u_{0} \|_{\dot{H}^{s_{c}}}} \| u_{0} \|_{\dot{H}^{1}}^{2} + \| u_{1} \|_{L^{2}}^{2}.
\end{equation}
Conservation of energy then implies a uniform bound on the $\| (u(t), u_{t}(t)) \|_{\dot{H}^{1} \times L^{2}}$ norm for the entire time of existence of $u$, which by Theorem $\ref{t1.1}$ implies that the solution to $(\ref{1.1})$ with initial data $u_{0} \in \dot{H}^{1} \cap \dot{H}^{s_{c}}$ and $u_{1} \in L^{2} \cap \dot{H}^{s_{c} - 1}$ is global.

For generic initial data $u_{0} \in B_{1,1}^{s_{c} + \frac{3}{2}}$ and $u_{1} \in B_{1,1}^{s_{c} + \frac{1}{2}}$, there is no reason to think that the initial data lies in $\dot{H}^{1} \times L^{2}$. However, using the dispersive estimate $(\ref{1.9})$, we can split a solution $u(t)$ into a piece lying in $\dot{H}^{1} \times L^{2}$ and a piece with good decay estimates as $t$ becomes large. A similar computation was used in \cite{dodson2018globalAPDE} to prove Theorem $\ref{t1.5}$.

The local well-posedness result of Theorem $\ref{t1.1}$ implies that there exists an open neighborhood $I$ of $0$ for which $(\ref{1.1})$ has a solution, and
\begin{equation}\label{1.13}
\| u \|_{L_{t,x}^{2(p - 1)}(I \times \mathbb{R}^{3})} \leq \epsilon,
\end{equation}
for some $\epsilon > 0$ small. Rescaling by $(\ref{1.2})$,
\begin{equation}\label{1.14}
\| u \|_{L_{t,x}^{2(p - 1)}([-1, 1] \times \mathbb{R}^{3})} \leq \epsilon.
\end{equation}
This solution satisfies Duhamel's principle
\begin{equation}\label{1.15}
u(t) = S(t)(u_{0}, u_{1}) - \int_{0}^{t} S(t - \tau)(0, |u|^{p - 1} u) d\tau.
\end{equation}

Next, combining the dispersive estimate $(\ref{1.9})$ and local well-posedness theory, it is possible to prove that 
\begin{equation}
t^{\frac{2 - s_{c}}{p}} \| u(t) \|_{L^{2p}},
\end{equation}
is uniformly bounded for all $t \in [-1, 1]$. Therefore, by standard energy estimates,
\begin{equation}\label{1.16}
\| \int_{1/2}^{1} S(1 - \tau)(0, |u|^{p - 1} u) d\tau \|_{\dot{H}^{1} \times L^{2}} \lesssim 1,
\end{equation}
with implicit constant bounded by the norm of the initial data in $B_{1,1}^{s_{c} + \frac{3}{2}} \times B_{1,1}^{s_{c} + \frac{1}{2}}$.

Let
\begin{equation}\label{1.17}
v(1) = \int_{1/2}^{1} S(1 - \tau)(0, |u|^{p - 1} u) d\tau, \qquad v_{t}(1) = \partial_{t} \int_{1/2}^{t} S(t - \tau)(0, |u|^{p - 1} u) d\tau|_{t = 1},
\end{equation}
and let
\begin{equation}\label{1.18}
w(1) = u(1) - v(1), \qquad w_{t}(1) = u_{t}(1) - v_{t}(1).
\end{equation}
It follows from $(\ref{1.14})$ and Theorem $\ref{t1.1}$ that $(\ref{1.1})$ has a local solution on $[1, T)$ for some $T > 1$. Decompose this solution $u = v + w$, which solve
\begin{equation}\label{1.19}
\aligned
w_{tt} - \Delta w = 0, \qquad w(1,x) = w(1), \qquad w_{t}(1, x) = w_{t}(1), \\
v_{tt} - \Delta v + u^{3} = 0, \qquad v(1,x) = v(1), \qquad v_{t}(1,x) = v_{t}(1).
\endaligned
\end{equation}
To prove that $T$ may be extended to $T = \infty$, it is enough to prove that $E(v(t))$, where $E$ is given by $(\ref{1.10})$, is uniformly bounded on any compact subset of $[1, \infty)$. To see why, first note that $w_{tt} - \Delta w = 0$ has a global solution. Next, the rescaling used to obtain $(\ref{1.14})$ will be used to show that for any $T \geq 0$,
\begin{equation}\label{1.20}
\| w \|_{L_{t,x}^{2(p - 1)}([T, T + 1] \times \mathbb{R}^{3})} \leq \frac{\epsilon}{2}.
\end{equation}
Therefore, using standard perturbative arguments,
\begin{equation}\label{1.21}
v_{tt} - \Delta v + |u|^{p - 1} u = 0,
\end{equation}
may be treated as a perturbation of
\begin{equation}\label{1.22}
v_{tt} - \Delta v + |v|^{p - 1} v = 0,
\end{equation}
on short time intervals. Therefore, if $E(v(t_{0})) < \infty$, $(\ref{1.1})$ is locally well-posed on the interval $[t_{0}, t_{0} + \frac{1}{E(v(t_{0}))}]$, so it is enough to prove that $E(v(t))$ is uniformly bounded on any compact subset of $[1, \infty)$.

To prove the uniform bound, standard calculations imply
\begin{equation}\label{1.23}
\frac{d}{dt} E(v(t)) = -\langle v_{t}, |u|^{p - 1} u - |v|^{p - 1} v \rangle.
\end{equation}
The most difficult component of $(\ref{1.23})$ is a term of the form
\begin{equation}\label{1.24}
-\langle v_{t}, v^{p - 1} w \rangle  \lesssim \| |\nabla|^{s_{c} - \frac{1}{2}} w \|_{L^{\infty}} E(v(t)).
\end{equation}
Using the dispersive estimate $(\ref{1.9})$ it is possible to prove $\| |\nabla|^{s_{c} - \frac{1}{2}} w \|_{L^{\infty}} \lesssim \frac{1}{t}$. Plugging this estimate into $(\ref{1.24})$ and using Gronwall's inequality then proves a uniform bound on $E(v(t))$ on any compact subset, completing the proof of global well-posedness.

The above computations are not enough to prove scattering. In fact, even if one assumed initial data $u_{0} \in \dot{H}^{1} \cap \dot{H}^{s_{c}}$ and $u_{1} \in L^{2} \cap \dot{H}^{s_{c} - 1}$, conservation of energy would not guarantee a uniform bound on $\| u(t), u_{t}(t) \|_{\dot{H}^{s_{c}} \times \dot{H}^{s_{c} - 1}}$. Indeed, recall that \cite{strauss1968decay} assumed sufficient decay on the initial data.

However, the Lebesgue dominated convergence theorem implies that outside a compact set, the initial data has small $\dot{H}^{s_{c}} \times \dot{H}^{s_{c} - 1}$ norm. By finite propagation speed, this implies scattering outside a light cone. Inside the light cone, we follow \cite{shen2014energy}, \cite{shen2017scattering}, \cite{dodson2018globalAPDE}, \cite{dodson2018global}, \cite{dodson2018global2} and make a conformal change of coordinates to prove that this solution scatters.

We obtain the bound $(\ref{1.8.1})$ using the profile decomposition argument in \cite{ramos2012refinement}.

\section{Local behavior of the solution to $(\ref{1.1})$}
Using $(\ref{1.2})$, it is possible to rescale equation $(\ref{1.1})$ so that $(\ref{1.1})$ is locally well-posed on $[-1, 1]$ and the solution satisfies
\begin{equation}\label{2.1}
\| u \|_{L_{t,x}^{2(p - 1)}([-1, 1] \times \mathbb{R}^{3})} \leq \epsilon.
\end{equation}
\begin{proof}[Proof of $\ref{2.1}$]Recall the Strichartz estimates for the wave equation.
\begin{theorem}\label{t2.1}
Let $I$ be a time interval and let $u : I \times \mathbb{R}^{3} \rightarrow \mathbb{R}$ be a Schwartz solution to the wave equation 
\begin{equation}\label{2.2}
u_{tt} - \Delta u = F, \qquad u(0) = u_{0}, \qquad \partial_{t} u(0) = u_{1},
\end{equation}
where $0 \in I$. Then we have the estimates,
\begin{equation}\label{2.3}
\| u \|_{L_{t}^{q} L_{x}^{r}(I \times \mathbb{R}^{3})} + \| u \|_{C_{t}^{0} \dot{H}_{x}^{s}(I \times \mathbb{R}^{3})} + \| \partial_{t} u \|_{C_{t}^{0} \dot{H}_{x}^{s}(I \times \mathbb{R}^{3})} \lesssim_{q, r, s} (\| u_{0} \|_{\dot{H}^{s}(\mathbb{R}^{3})} + \| u_{1} \|_{\dot{H}^{s - 1}(\mathbb{R}^{3})} + \| F \|_{L_{t}^{\tilde{q}'} L_{x}^{\tilde{r}'}(I \times \mathbb{R}^{3})}),
\end{equation}
for any $s \geq 0$, $2 < q, \tilde{q} \leq \infty$, and $2 \leq r, \tilde{r} < \infty$ obey the scaling condition,
\begin{equation}\label{2.4}
\frac{1}{q} + \frac{3}{r} = \frac{3}{2} - s = \frac{1}{\tilde{q}'} + \frac{3}{\tilde{r}'} - 2,
\end{equation}
and satisfy the wave admissibility conditions
\begin{equation}\label{2.5}
\frac{1}{q} + \frac{1}{r}, \qquad \frac{1}{\tilde{q}} + \frac{1}{\tilde{r}} \leq \frac{1}{2}.
\end{equation}
\end{theorem}
\begin{proof}
This theorem is copied from \cite{tao2006nonlinear}. See \cite{strichartz1977restrictions}, \cite{kato1994q}, \cite{ginibre1995generalized}, \cite{kapitanski1989some}, \cite{lindblad1995existence}, \cite{sogge1995lectures}, \cite{shatah2000geometric}, \cite{keel1998endpoint} for the proof of this theorem.
\end{proof}

By Theorem $\ref{t2.1}$, if $u$ solves $(\ref{2.1})$, then
\begin{equation}\label{2.6}
\| u \|_{L_{t}^{q} L_{x}^{r} \cap L_{t,x}^{2(p - 1)}([-1,1] \times \mathbb{R}^{3})} \lesssim_{p} \| u_{0} \|_{\dot{H}^{s_{c}}} + \| u_{1} \|_{\dot{H}^{s_{c} - 1}} + \| F \|_{L_{t}^{\tilde{q}'} L_{x}^{\tilde{r}'}([-1,1] \times \mathbb{R}^{3})},
\end{equation}
where
\begin{equation}\label{2.7}
\frac{1}{q} = \frac{1}{2} s_{c}, \qquad \frac{1}{r} = \frac{1}{2} - \frac{1}{2} s_{c}, \qquad s_{c} = \frac{3}{2} - \frac{2}{p - 1}, \qquad \frac{1}{\tilde{q}'} = \frac{1}{q} + \frac{1}{2}, \qquad \frac{1}{\tilde{r}'} = \frac{1}{r} + \frac{1}{2}.
\end{equation}
When $3 < p < 5$, $(q, r)$ is an admissible pair that satisfies $(\ref{2.5})$, and $\tilde{q}'$ and $\tilde{r}'$ satisfies $(\ref{2.4})$.

Since $(u_{0}, u_{1}) \in B_{1,1}^{s_{c} + \frac{3}{2}} \times B_{1,1}^{s_{c} + \frac{1}{2}}$, there exists some $j_{0} \in \mathbb{Z}$ such that
\begin{equation}\label{2.8}
\sum_{j \geq j_{0}} 2^{j(s_{c} + \frac{3}{2})} \| P_{j} u_{0} \|_{L^{1}} + 2^{j(s_{c} + \frac{1}{2})} \| P_{j} u_{1} \|_{L^{1}} \leq c \epsilon,
\end{equation}
for some $c > 0$ that is determined by the implicit constant in $(\ref{2.3})$. Using $(\ref{1.2})$, rescale so that 
\begin{equation}\label{2.9}
2^{j_{0}(1 - s_{c})} \cdot \| (u_{0}, u_{1}) \|_{B_{1,1}^{s_{c} + \frac{3}{2}} \times B_{1,1}^{s_{c} + \frac{1}{2}}} \leq c \epsilon.
\end{equation}
Theorem $\ref{t2.1}$ and $(\ref{2.8})$ imply
\begin{equation}\label{2.10}
\| S(t)(P_{\geq j_{0}} u_{0}, P_{\geq j_{0}} u_{1}) \|_{L_{t}^{q} L_{x}^{r} \cap L_{t,x}^{2(p - 1)}(\mathbb{R} \times \mathbb{R}^{3})} \leq \frac{\epsilon}{4}.
\end{equation}
Also, by the Sobolev embedding theorem, $(\ref{2.8})$, and the fact that $S(t)$ is a unitary operator on $\dot{H}^{s} \times \dot{H}^{s - 1}$,
\begin{equation}\label{2.11}
\| S(t)(P_{\leq j_{0}} u_{0}, P_{\leq j_{0}} u_{1}) \|_{L_{t}^{\infty} L^{2(p - 1)}(\mathbb{R} \times \mathbb{R}^{3})} \leq \frac{\epsilon}{4},
\end{equation}
so by H{\"o}lder's inequality,
\begin{equation}\label{2.12}
\| S(t)(u_{0}, u_{1}) \|_{L_{t,x}^{2(p - 1)}([-1, 1] \times \mathbb{R}^{3})} \leq \frac{3\epsilon}{4}.
\end{equation}
A similar calculation also implies
\begin{equation}\label{2.12.1}
\| S(t)(u_{0}, u_{1}) \|_{L_{t}^{q} L_{x}^{r}([-1, 1] \times \mathbb{R}^{3})} \leq \frac{3\epsilon}{4}.
\end{equation}
Plugging $(\ref{2.12})$ into $(\ref{1.15})$ and using $(\ref{2.3})$ and Picard iteration implies that for $\epsilon > 0$ sufficiently small, $(\ref{1.1})$ is locally well-posed on $[-1, 1]$, and the solution satisfies
\begin{equation}\label{2.13}
\| u \|_{L_{t,x}^{2(p - 1)}([-1,1] \times \mathbb{R}^{3})} \leq \epsilon.
\end{equation}
See \cite{lindblad1995existence} for a detailed proof.
\end{proof}

The constant $\epsilon > 0$ will eventually be chosen to depend on $\| u_{0} \|_{B_{1,1}^{s_{c} + \frac{3}{2}}} + \| u_{1} \|_{B_{1,1}^{s_{c} + \frac{1}{2}}}$. Under $(\ref{2.1})$, the behavior of $u$ on the interval $[-1, 1]$ is approximately linear.
\begin{theorem}\label{t2.2}
If $u$ is a solution to $(\ref{1.1})$ on $[-1, 1]$ with $\| u \|_{L_{t,x}^{2(p - 1)}([-1,1] \times \mathbb{R}^{3})} \leq \epsilon(A)$, where $(u_{0}, u_{1}) \in B_{1,1}^{3/2 + s_{c}} \times B_{1,1}^{1/2 + s_{c}}$ with $A = \| u_{0} \|_{B_{1,1}^{3/2 + s_{c}}} + \| u_{1} \|_{B_{1,1}^{1/2 + s_{c}}}$, then
\begin{equation}\label{2.14}
\sum_{j} 2^{j s_{c}} \| P_{j} u \|_{L_{t}^{\infty} L_{x}^{2}([-1, 1] \times \mathbb{R}^{3})} \lesssim A.
\end{equation}
\end{theorem}
\begin{proof}

Using the Strichartz estimates in Theorem $\ref{t2.1}$, if $(q, r)$ and $(\tilde{q}, \tilde{r})$ are given by $(\ref{2.7})$,
\begin{equation}\label{2.15}
\aligned
2^{j s_{c}} \| P_{j} u \|_{L_{t}^{\infty} L_{x}^{2}([-1,1] \times \mathbb{R}^{3})} + \| P_{j} u \|_{L_{t,x}^{2(p - 1)} \cap L_{t}^{q} L_{x}^{r}([-1,1] \times \mathbb{R}^{3})} + 2^{j(s_{c} - \frac{1}{4})} \| P_{j} u \|_{L_{t,x}^{4}([-1, 1] \times \mathbb{R}^{3})} \\ + 2^{-j(1 - s_{c})/2} \| P_{j} u \|_{L_{t}^{\frac{4q}{2 + q}} L_{x}^{2r}([-1,1] \times \mathbb{R}^{3})}
 \lesssim 2^{j s_{c}} \| P_{j} u_{0} \|_{L^{2}} + 2^{j(s_{c} - 1)} \| P_{j} u_{1} \|_{L^{2}} \\ + 2^{-j (1 - s_{c})/2} \| P_{j} F_{1} \|_{L_{t}^{\frac{4q}{3q + 2}} L_{x}^{\frac{2r}{r + 1}}([-1,1] \times \mathbb{R}^{3})} + 2^{j(s_{c} - \frac{1}{4})} \| P_{j} F_{2} \|_{L_{t}^{8/5} L_{x}^{8/7}([-1, 1] \times \mathbb{R}^{3})},
\endaligned
\end{equation}
where $P_{j} F_{1} + P_{j} F_{2} = P_{j}(|u|^{p - 1} u)$ is a decomposition of the nonlinearity. Using Taylor's theorem, decompose
\begin{equation}\label{2.16}
F_{1} = |P_{\leq j} u|^{p - 1} (P_{\leq j} u), \qquad F_{2} = |u|^{p - 1} u - |P_{\leq j} u|^{p - 1} (P_{\leq j} u) = O(|u|^{p - 1} |P_{\geq j} u|).
\end{equation}
Theorem $\ref{t2.2}$ follows directly from $(\ref{2.15})$ and $u_{0} \in B_{1,1}^{s_{c} + 3/2}$, $u_{1} \in B_{1,1}^{s_{c} + 1/2}$. Indeed,
\begin{equation}\label{2.17}
\| F_{1} \|_{L_{t}^{\frac{4q}{3q + 2}} L_{x}^{\frac{2r}{r + 1}}([-1,1] \times \mathbb{R}^{3})} \lesssim \| P_{\leq j} u \|_{L_{t,x}^{2(p - 1)}([-1,1] \times \mathbb{R}^{3})}^{p - 1} \| P_{\leq j} u \|_{L_{t}^{\frac{4q}{q + 2}} L_{x}^{2r}([-1, 1] \times \mathbb{R}^{3})},
\end{equation}
and
\begin{equation}\label{2.18}
\| F_{2} \|_{L_{t}^{8/5} L_{x}^{8/7}([-1,1] \times \mathbb{R}^{3})} \lesssim \| P_{\geq j} u \|_{L_{t,x}^{4}([-1,1] \times \mathbb{R}^{3})} \| u \|_{L_{t,x}^{2(p - 1)}([-1, 1] \times \mathbb{R}^{3})}^{p - 1},
\end{equation}
so by Young's inequality and $(\ref{2.15})$, the proof of Theorem $\ref{t2.2}$ is complete. Indeed, letting $X_{j}$ denote the left hand side of $(\ref{2.15})$,
\begin{equation}
X_{j} \lesssim 2^{j s_{c}} \| P_{j} u_{0} \|_{L^{2}} + 2^{j(s_{c} - 1)} \| P_{j} u_{1} \|_{L^{2}} + \epsilon^{p - 1} \sum_{k \geq j} 2^{(j - k)(s_{c} - \frac{1}{4})} X_{k} + \epsilon^{p - 1} \sum_{k \leq j} 2^{(k - j) \frac{(1 - s_{c})}{2}} X_{k},
\end{equation}
which implies $(\ref{2.14})$.
\end{proof}

The dispersive estimates $(\ref{1.9})$ also give additional $L_{t}^{q} L_{x}^{r}$ bounds on the solution $u$ in $[-1,1]$ that lie outside the admissible pairs in Theorem $\ref{t2.1}$.
\begin{theorem}\label{t2.3}
For $3 < p < 5$, if $\frac{1}{q} = \frac{3}{2} - s_{c} = \frac{2}{p - 1}$,
\begin{equation}\label{2.19}
\| u \|_{L_{t}^{q} L_{x}^{\infty}([-1,1] \times \mathbb{R}^{3})} + \sum_{j} 2^{j(s_{c} - \frac{1}{2})} \sup_{t \in [-1, 1]} t \| P_{j} u \|_{L^{\infty}} \lesssim \epsilon.
\end{equation}
\end{theorem}
\begin{proof}
Using the dispersive estimate
\begin{equation}\label{2.20}
\| S(t)(u_{0}, u_{1}) \|_{L^{\infty}} \lesssim \frac{1}{t} \| (u_{0}, u_{1}) \|_{B_{1,1}^{2} \times B_{1,1}^{1}},
\end{equation}
for any $j \in \mathbb{Z}$,
\begin{equation}\label{2.21}
\| S(t)(P_{j} u_{0}, P_{j} u_{1}) \|_{L^{\infty}} \lesssim  \frac{1}{t} 2^{-j(s_{c} - \frac{1}{2})} [2^{j (\frac{3}{2} + s_{c})} \| P_{j} u_{0} \|_{L^{1}} + 2^{j(\frac{1}{2} + s_{c})} \| P_{j} u_{1} \|_{L^{1}}].
\end{equation}
Interpolating $(\ref{2.21})$ with
\begin{equation}\label{2.22}
\| S(t)(P_{j} u_{0}, P_{j} u_{1}) \|_{L^{\infty}} \lesssim 2^{j(\frac{3}{2} - s_{c})} \| (P_{j} u_{0}, P_{j} u_{1}) \|_{\dot{H}^{s_{c}} \times \dot{H}^{s_{c} - 1}},
\end{equation}
and making use of $(\ref{2.8})$ and $(\ref{2.10})$, we have proved
\begin{equation}\label{2.23}
\aligned
\sum_{j} \sup_{t \in [-1,1]} t^{\frac{3}{2} - s_{c}} \| S(t)(P_{j} u_{0}, P_{j} u_{1}) \|_{L^{\infty}} + \sum_{j} \| S(t) (P_{j} u_{0}, P_{j} u_{1}) \|_{L_{t}^{q} L_{x}^{\infty}(\mathbb{R} \times \mathbb{R}^{3})} \\ + \sum_{j} 2^{j(s_{c} - \frac{1}{2})} \sup_{t \in [-1,1]} t \| S(t)(P_{j} u_{0}, P_{j} u_{1}) \|_{L^{\infty}} \lesssim \epsilon.
\endaligned
\end{equation}

Turning to the second term in $(\ref{1.15})$ and using the formula for the solution to the linear wave equation in $\mathbb{R}^{3}$, see for example \cite{sogge1995lectures}, for any $x \in \mathbb{R}^{3}$,
\begin{equation}\label{2.24}
| S(t - \tau) (0, |u|^{p - 1} u)(x) | \lesssim \frac{1}{|t - \tau|} \int_{\partial B(x, t - \tau)} |u(y,\tau)|^{p} d\sigma(y).
\end{equation}
Once again, split
\begin{equation}\label{2.25}
P_{j}(|u|^{p - 1} u) = P_{j} F_{1} + P_{j} F_{2}, \qquad F_{1} = |P_{\leq j} u|^{p - 1} (P_{\leq j} u), \qquad F_{2} = O(|P_{\geq j} u| |u|^{p - 1}).
\end{equation}
Plugging $F_{2}$ into $(\ref{2.24})$, for any $t \in [-1,1]$, $x \in \mathbb{R}^{3}$,
\begin{equation}\label{2.26}
\aligned
| \int_{0}^{\frac{t}{2}} S(t - \tau)(0, P_{j} F_{2})(t,x) d\tau| \lesssim \frac{1}{t} \| |u|^{\frac{p - 1}{2}} \|_{L_{\tau}^{1} L_{x}^{\infty}([0, \frac{t}{2}] \times \mathbb{R}^{3})} \\ \cdot \sup_{\tau \in [0, \frac{t}{2}]} (\int_{\partial B(x, t - \tau)} |u(\tau, y)|^{p - 1} d\sigma(y))^{1/2} \cdot \sup_{\tau \in [0, \frac{t}{2}]} (\int_{\partial B(x, t - \tau)} |P_{\geq j} u(\tau, y)|^{2} dy)^{1/2}.
\endaligned
\end{equation}

By an argument similar to the Sobolev embedding theorem, for any $k \in \mathbb{Z}$,
\begin{equation}\label{2.27}
\int_{\partial B(x, t - \tau)} |P_{k} u(y, \tau)|^{p - 1} d\sigma(y) \lesssim 2^{k} \| P_{k} u \|_{L^{p - 1}}^{p - 1}.
\end{equation}
\begin{remark}
To see why this is so, recall that the Littlewood--Paley kernel for $P_{k}$ may be approximated by $2^{3k}$ multiplied by the characteristic function of a ball of radius $2^{-k}$. Then consider the cases when $2^{-k} \leq |t - \tau|$ and $2^{-k} > |t - \tau|$ separately. Indeed, for $|t - \tau| \lesssim 2^{-k}$, there exists some $C$ such that
\begin{equation}\label{2.28}
\int_{\partial B(x, t - \tau)} |P_{k} u(y, \tau)|^{p - 1} d\sigma(y) \lesssim 2^{3k} |t - \tau|^{2} \int_{B(x, C 2^{-k})} |P_{k} u(\tau, y)|^{p - 1} dy \lesssim 2^{k} \| P_{k} u(\tau) \|_{L^{p - 1}}^{p - 1}.
\end{equation}
Meanwhile, for $|t - \tau| \gg 2^{-k}$,
\begin{equation}\label{2.29}
\int_{\partial B(x, t - \tau)} |P_{k} u(y, \tau)|^{p - 1} d\sigma(y) \lesssim 2^{k}  \int_{dist(B(x, t - \tau), y) \leq 2^{-k}} |P_{k} u(\tau, y)|^{p - 1} dy \lesssim 2^{k} \| P_{k} u(\tau) \|_{L^{p - 1}}^{p - 1}.
\end{equation}
Now, then, since the Littlewood--Paley kernel obeys the bounds
\begin{equation}\label{2.30}
\mathcal F(P_{k}(y)) \lesssim_{N} 2^{3k} (1 + 2^{k} |y|)^{-N},
\end{equation}
for any $N$, calculations similar to $(\ref{2.28})$ and $(\ref{2.29})$ imply $(\ref{2.27})$.
\end{remark}

Plugging $(\ref{2.27})$ into $(\ref{2.26})$, by Young's inequality,
\begin{equation}\label{2.31}
\sum_{j} 2^{j(s_{c} - \frac{1}{2})} \sup_{t \in [-1,1]} t \| \int_{0}^{\frac{t}{2}} S(t - \tau)(0, P_{j}F_{2}) d\tau \|_{L^{\infty}} \lesssim \| u \|_{L_{t}^{\frac{p - 1}{2}} L_{x}^{\infty}([-1,1] \times \mathbb{R}^{3})}^{\frac{p - 1}{2}} A^{\frac{p + 1}{2}}.
\end{equation}
Meanwhile, since by Bernstein's inequality,
\begin{equation}\label{2.32}
P_{j}(F_{1}) \sim 2^{-j} \nabla P_{j} F_{1} \sim 2^{-j} |P_{\leq j} u|^{p - 1} |\nabla P_{\leq j} u|,
\end{equation}
\begin{equation}\label{2.33}
\aligned
| \int_{0}^{\frac{t}{2}} S(t - \tau)(0, P_{j} F_{1})(t,x) d\tau| \lesssim \frac{2^{-j}}{t} \| |u|^{\frac{p - 1}{2}} \|_{L_{\tau}^{1} L_{x}^{\infty}([0, \frac{t}{2}] \times \mathbb{R}^{3})} \\ \cdot \sup_{\tau \in [0, \frac{t}{2}]} (\int_{B(x, t - \tau)} |u(\tau, y)|^{p - 1} d\sigma(y))^{1/2} \cdot \sup_{\tau \in [0, \frac{t}{2}]} (\int_{B(x, t - \tau)} |\nabla P_{\leq j} u(\tau, y)|^{2} dy)^{1/2},
\endaligned
\end{equation}
and therefore,
\begin{equation}\label{2.34}
\sum_{j} 2^{j(s_{c} - \frac{1}{2})} \sup_{t \in [-1,1]} t \| \int_{0}^{\frac{t}{2}} S(t - \tau)(0, P_{j} |u|^{p - 1} u) d\tau \|_{L^{\infty}} \lesssim \| u \|_{L_{t}^{\frac{p - 1}{2}} L_{x}^{\infty}([-1,1] \times \mathbb{R}^{3})}^{\frac{p - 1}{2}} A^{\frac{p + 1}{2}}.
\end{equation}

For $\tau \in [\frac{t}{2}, t]$, energy estimates and the Sobolev embedding theorem imply,
\begin{equation}\label{2.35}
\aligned
\| S(t - \tau) (0, P_{j}(|u|^{p - 1} u)) \|_{L^{6}} \lesssim \frac{1}{t^{2}} (\sup_{\tau \in [\frac{t}{2}, t]} \tau \cdot \| |u(\tau)|^{\frac{p - 1}{2}} \|_{L_{x}^{\infty}(\mathbb{R}^{3})})^{2} \sup_{\tau \in [\frac{t}{2}, t]} \|P_{\geq j} u \|_{L^{2}} \\
+ \frac{2^{-j}}{t^{2}} (\sup_{\tau \in [\frac{t}{2}, t]} \tau \cdot \| |u(\tau)|^{\frac{p - 1}{2}} \|_{L_{x}^{\infty}(\mathbb{R}^{3})})^{2}  \| P_{\leq j} \nabla u \|_{L^{2}}.
\endaligned
\end{equation}
Therefore, by Young's inequality, the Sobolev embedding theorem, and Theorem $\ref{t2.2}$,
\begin{equation}\label{2.36}
\sum_{j} 2^{j(s_{p} - 1/2)} \sup_{t \in [-1,1]} t \| \int_{t/2}^{t} S(t - \tau) P_{j}(0, |u|^{p - 1} u) d\tau \|_{L^{\infty}} \lesssim (\sup_{t \in [-1,1]} t^{\frac{3}{2} - s_{c}} \| u(t) \|_{L^{\infty}})^{p - 1} A.
\end{equation}

Combining $(\ref{2.21})$, $(\ref{2.34})$, and $(\ref{2.36})$,
\begin{equation}\label{2.37}
\sum_{j} 2^{j(s_{p} - \frac{1}{2})} \sup_{t \in [-1,1]} t \| P_{j} u(t) \|_{L^{\infty}} \lesssim \epsilon + \| u \|_{L_{t}^{\frac{p - 1}{2}} L_{x}^{\infty}([-1,1] \times \mathbb{R}^{3})}^{\frac{p - 1}{2}} A^{\frac{p + 1}{2}} + (\sup_{t \in [-1,1]} t^{\frac{3}{2} - s_{c}} \| u(t) \|_{L^{\infty}})^{p - 1} A.
\end{equation}
Now then, for any $3 < p < 5$, Theorem $\ref{t2.2}$, $(\ref{2.10})$, $(\ref{2.11})$, and the Sobolev embedding theorem imply
\begin{equation}\label{2.37.1}
\sum_{j} \| P_{j} u \|_{L_{t}^{\frac{p - 1}{2}} L_{x}^{\infty}([-1, 1] \times \mathbb{R}^{3})} + \sum_{j} \sup_{t \in [-1,1]} t^{\frac{3}{2} - s_{c}} \| P_{j} u \|_{L_{x}^{\infty}(\mathbb{R}^{3})} \lesssim \epsilon^{\frac{p - 3}{p - 1}} (\sum_{j} 2^{j(s_{p} - \frac{1}{2})} \sup_{t \in [-1,1]} t \| P_{j} u(t) \|_{L^{\infty}})^{\frac{2}{p - 1}}.
\end{equation}
Combining $(\ref{2.37})$ with $(\ref{2.37.1})$ proves the Theorem.
\end{proof}

Theorem $\ref{t2.3}$ implies finite energy for a piece of the Duhamel term.
\begin{corollary}\label{c2.4}
For any $t \in [-1,1]$,
\begin{equation}\label{2.38}
\int_{t/2}^{t} \| u^{p}(\tau) \|_{L^{2}} d\tau \lesssim \frac{A^{p}}{t^{1 - s_{p}}}.
\end{equation}
\end{corollary}
\begin{proof}
Use the energy estimate in $(\ref{2.35})$.
\end{proof}

\section{Proof of Global well-posedness}
By time reversal symmetry and local well-posedness on the interval $[-1,1]$, to prove Theorem $\ref{t1.6}$, it suffices to prove global well-posedness in the positive time direction, $t > 1$ for $(\ref{1.1})$ with initial data $(u(1,x), u_{t}(1,x))$. The local well-posedness arguments used to prove Theorem $\ref{t1.1}$ imply that $(\ref{1.1})$ has a solution on some open interval $[0, T)$ for some $T > 1$, so to prove global well-posedness it suffices to show that $T$ can be taken to go to infinity.

Split
\begin{equation}\label{3.1}
\begin{pmatrix} u(1,x) \\ u_{t}(1,x) \end{pmatrix} = S(1) (u_{0}, u_{1}) + \int_{0}^{1/2} S(1 - \tau)(0, |u|^{p - 1} u) d\tau + \int_{1/2}^{1} S(1 - \tau)(0, |u|^{p - 1} u) d\tau.
\end{equation}
By Corollary $\ref{c2.4}$, the second Duhamel term has finite energy.
\begin{equation}\label{3.2}
\| \begin{pmatrix} v(1, x) \\ v_{t}(1, x) \end{pmatrix} \|_{\dot{H}^{1} \times L^{2}} = \| \int_{1/2}^{1} S(1 - \tau)(0, |u|^{p - 1} u) d\tau \|_{\dot{H}^{1} \times L^{2}}  \lesssim_{A} 1.
\end{equation}

Now let $u$ be the solution to $(\ref{1.1})$ on $[1, T)$. Split $u = v + w$, where $v$ solves
\begin{equation}\label{3.3}
v_{tt} - \Delta v + |u|^{p - 1} u = 0,
\end{equation}
on $[1, T)$ with initial data given by $(\ref{3.2})$, and
\begin{equation}\label{3.4}
w_{tt} - \Delta w = 0, \qquad w(1,x) = u(1,x) - v(1,x), \qquad w_{t}(1,x) = u_{t}(1,x) - v_{t}(1,x).
\end{equation}
Set
\begin{equation}\label{3.4.1}
E(v) = \int [\frac{1}{2} |v_{t}|^{2} + \frac{1}{2} |\nabla v|^{2} + \frac{1}{p + 1} |v|^{p + 1}] dx,
\end{equation}
and compute
\begin{equation}\label{3.5}
\frac{d}{dt} E(v) = \langle v_{t}, -|u|^{p - 1} u + |v|^{p - 1} v \rangle.
\end{equation}

By Taylor's theorem,
\begin{equation}\label{3.6}
|u|^{p - 1} u - |v|^{p - 1} v = p |v|^{p - 1} w + O(|w|^{2} |v|^{p - 2}) + O(|w|^{p}).
\end{equation}
By H{\"o}lder's inequality,
\begin{equation}\label{3.7}
\langle O(|w|^{2} |v|^{p - 2}), v_{t} \rangle \lesssim \| v_{t} \|_{L^{2}} \| v \|_{L^{p + 1}}^{p - 2} \| w \|_{L^{\infty}}^{\frac{p - 1}{2}} \| w \|_{L^{p + 1}}^{\frac{5 - p}{2}} \lesssim E(v(t))^{\frac{1}{2} + \frac{p - 2}{p + 1}} \| w \|_{L^{\infty}}^{\frac{p - 1}{2}} \| w \|_{L^{p + 1}}^{\frac{5 - p}{2}}.
\end{equation}
Interpolating $(\ref{2.23})$ with $\| w \|_{L^{\frac{3(p - 1)}{2}}} \lesssim \| w \|_{\dot{H}^{s_{c}}} \lesssim_{A} 1$, proves $\| w \|_{L^{p + 1}} \lesssim_{A} 1$. Also,
\begin{equation}\label{3.8}
\langle |w|^{p}, v_{t} \rangle \lesssim \| v_{t} \|_{L^{2}} \| w \|_{L^{\infty}}^{\frac{p - 1}{2}} \| w \|_{L^{p + 1}}^{\frac{p + 1}{2}} \lesssim E(v(t))^{1/2} \| w \|_{L^{\infty}}^{\frac{p - 1}{2}} \| w \|_{L^{p + 1}}^{\frac{p + 1}{2}}.
\end{equation}
If we could ignore the term
\begin{equation}\label{3.9}
\langle v_{t}, p |v|^{p - 1} w \rangle,
\end{equation}
then $E(v(t))$ would be uniformly bounded on $\mathbb{R}$ by Gronwall's inequality. Indeed, by $(\ref{2.19})$,
\begin{equation}\label{3.9.1}
\aligned
\int_{1}^{T} E(v(t))^{\frac{1}{2} + \frac{p - 2}{p + 1}} \| w \|_{L^{\infty}}^{\frac{p - 1}{2}} \| w \|_{L^{p + 1}}^{\frac{5 - p}{2}} dt + \int E(v(t))^{1/2} \| w \|_{L^{\infty}}^{\frac{p - 1}{2}} \| w \|_{L^{p + 1}}^{\frac{p + 1}{2}} dt \\ \lesssim \sup_{t \in [1, T)} \epsilon E(v(t))^{\frac{1}{2} + \frac{p - 2}{p + 1}} +\sup_{t \in [1, T)} \epsilon E(v(t))^{1/2},
\endaligned
\end{equation}
which implies a uniform bound on $E(v(t))$.

To deal with the contribution of $(\ref{3.9})$, take the modified energy
\begin{equation}\label{3.10}
\mathcal E(v(t)) = E(v(t)) + \langle |v|^{p - 1} v, w \rangle.
\end{equation}
Then $(\ref{3.5})$ and $(\ref{3.6})$ imply
\begin{equation}\label{3.11}
\aligned
\frac{d}{dt} \mathcal E(v(t)) = \langle v_{t}, -|u|^{p - 1} u + |v|^{p - 1} v \rangle + \langle p |v|^{p - 1} w, v_{t} \rangle + \langle |v|^{p - 1} v, w_{t} \rangle \\
= \langle |v|^{p - 1} v, w_{t} \rangle + O(E(v(t))^{\frac{1}{2} + \frac{p - 2}{p + 1}} \| w \|_{L^{\infty}}^{\frac{2}{p - 1}} \| w \|_{L^{p + 1}}^{\frac{5 - p}{2}}) + O(E(v(t))^{1/2} \| w \|_{L^{\infty}}^{\frac{p - 1}{2}} \| w \|_{L^{p + 1}}^{\frac{p + 1}{2}}).
\endaligned
\end{equation}
Also,
\begin{equation}\label{3.12}
\langle |v|^{p - 1} v, w \rangle \lesssim \| v \|_{L^{p + 1}}^{p} \| w \|_{L^{p + 1}} \lesssim E(v(t))^{\frac{p}{p + 1}},
\end{equation}
so when $E(v(t))$ is large,
\begin{equation}\label{3.13}
E(v(t)) \sim \mathcal E(v(t)),
\end{equation}
and
\begin{equation}\label{3.13.1}
\frac{d}{dt} \mathcal E(v(t)) = \langle |v|^{p - 1} v, w_{t} \rangle + O(\mathcal E(v(t))^{\frac{1}{2} + \frac{p - 2}{p + 1}} \| w \|_{L^{\infty}}^{\frac{2}{p - 1}} \| w \|_{L^{p + 1}}^{\frac{5 - p}{2}}) + O(\mathcal E(v(t))^{1/2} \| w \|_{L^{\infty}}^{\frac{p - 1}{2}} \| w \|_{L^{p + 1}}^{\frac{p + 1}{2}}).
\end{equation}

Splitting $w_{t} = \sum_{j} P_{j} w_{t}$,
\begin{equation}\label{3.14}
\langle |v|^{p - 1} v, w_{t} \rangle = \sum_{j} \langle P_{j}(|v|^{p - 1} v), P_{j} w_{t} \rangle.
\end{equation}
Now by Bernstein's inequality and $(\ref{2.19})$,
\begin{equation}\label{3.15}
\sum_{j} \langle P_{j}(|v|^{p - 1} v - |P_{\leq j} v|^{p - 1} (P_{\leq j} v)), P_{j} w_{t} \rangle \lesssim \sum_{j} \| P_{j} w_{t} \|_{L^{\infty}} \| P_{\geq j} v \|_{L^{\frac{p + 1}{2}}} \| v \|_{L^{p + 1}}^{p - 1} \lesssim \frac{\epsilon}{t} E(v(t)).
\end{equation}

Indeed, by $(\ref{2.19})$,
\begin{equation}\label{3.15.1}
\sum_{j} 2^{j(s_{c} - \frac{1}{2})} 2^{-j} \| P_{j} w_{t} \|_{L^{\infty}} \lesssim \frac{\epsilon}{t}.
\end{equation}
Meanwhile, by Bernstein's inequality, for any fixed $j \in \mathbb{Z}$,
\begin{equation}\label{3.15.2}
\aligned
\frac{\epsilon}{t} 2^{j} 2^{-j(s_{c} - \frac{1}{2})} \| P_{\geq j} v \|_{L^{\frac{p + 1}{2}}} \| v \|_{L^{p + 1}}^{p - 1} \lesssim \frac{\epsilon}{t} 2^{j(\frac{3}{2} - s_{c})} \| P_{\geq j} v \|_{L^{2}}^{\frac{2}{p - 1}} \| v \|_{L^{p + 1}}^{\frac{p - 3}{p - 1} + p - 1} \lesssim \frac{\epsilon}{t} \| \nabla v \|_{L^{2}}^{\frac{2}{p - 1}} \| v \|_{L^{p + 1}}^{\frac{(p - 2)(p + 1)}{p - 1}} \lesssim \frac{\epsilon}{t} E(v(t)).
\endaligned
\end{equation}
Also, by Bernstein's inequality,
\begin{equation}\label{3.16}
\sum_{j} \langle P_{j}(|P_{\leq j} v|^{p - 1} (P_{\leq j} v), P_{j} w_{t} \rangle \lesssim \sum_{j} 2^{-j(\frac{2}{p - 1})} \| \nabla v \|_{L^{2}}^{\frac{2}{p - 1}} \| P_{j} w_{t} \|_{L^{\infty}} \| v \|_{L^{p + 1}}^{(p + 1) \frac{p - 2}{p - 1}} \lesssim \frac{\epsilon}{t} E(v(t)).
\end{equation}
Therefore, by Gronwall's inequality,
\begin{equation}\label{3.17}
\mathcal E(v(t)) < \infty, \qquad \text{and} \qquad E(v(t)) < \infty,
\end{equation}
for any $t \in \mathbb{R}$. This proves global well-posedness.

\section{Proof of scattering}
Now we prove that the global solution in the previous section scatters. By time reversal symmetry, to prove
\begin{equation}\label{4.1}
\| u \|_{L_{t,x}^{2(p - 1)}(\mathbb{R} \times \mathbb{R}^{3})} < \infty,
\end{equation}
if $u$ is a solution to $(\ref{1.1})$ with initial data $(u_{0}, u_{1}) \in B_{1,1}^{\frac{3}{2} + s_{c}} \times B_{1,1}^{\frac{1}{2} + s_{c}}$, it is enough to prove that
\begin{equation}\label{4.2}
\| u \|_{L_{t,x}^{2(p - 1)}([0, \infty) \times \mathbb{R}^{3})} < \infty.
\end{equation}
Recall from $(\ref{2.8})$ that
\begin{equation}\label{4.3}
\| P_{\geq j_{0}} u_{0} \|_{\dot{H}^{s_{c}}(\mathbb{R}^{3})} + \| P_{\geq j_{0}} u_{1} \|_{\dot{H}^{s_{c} - 1}(\mathbb{R}^{3})} \lesssim \epsilon.
\end{equation}
Also, let $\chi \in C_{0}^{\infty}(\mathbb{R}^{3})$ be a smooth, compactly supported function, and suppose $\chi(x) = 1$ for $|x| \leq 1$ and $\chi(x)$ is supported on $|x| \leq 2$. By the dominated convergence theorem there exists $100 \leq R(u_{0}, u_{1}, \epsilon) < \infty$ such that
\begin{equation}\label{4.4}
\| (1 - \chi(\frac{x}{R})) P_{\leq j_{0}} u_{0} \|_{\dot{H}^{s_{c}}(\mathbb{R}^{3})} + \| (1 - \chi(\frac{x}{R})) P_{\leq j_{0}} u_{1} \|_{\dot{H}^{s_{c} - 1}(\mathbb{R}^{3})} \lesssim \epsilon.
\end{equation}
Rescaling using $(\ref{1.2})$ and translating the initial data in time, $(\ref{4.2})$ is equivalent to proving
\begin{equation}\label{4.5}
\| u \|_{L_{t,x}^{2(p - 1)} [1 - \frac{1}{4R}, \infty) \times \mathbb{R}^{3}} < \infty,
\end{equation}
where
\begin{equation}\label{4.6}
u(1 - \frac{1}{4R}, x) = (4R)^{\frac{2}{p - 1}} u_{0}(4Rx), \qquad u_{t}(1 - \frac{1}{4R}, x) = (4R)^{1 + \frac{2}{p - 1}} u_{1}(4Rx).
\end{equation}
Then decompose
\begin{equation}\label{4.7}
\aligned
v(1 - \frac{1}{4R}, x) = \chi(\frac{x}{4}) (4R)^{\frac{2}{p - 1}} ( P_{\leq j_{0}} u_{0})(4Rx), \qquad v_{t}(1 - \frac{1}{4R}, x) = \chi(\frac{x}{4})(4R)^{\frac{2}{p - 1}} ( P_{\leq j_{0}} u_{1})(4Rx), \\
u(1 - \frac{1}{4R}, x) = v(1 - \frac{1}{4R}, x) + w(1 - \frac{1}{4R}, x), \qquad u_{t}(1 - \frac{1}{4R}, x) = v_{t}(1 - \frac{1}{4R}, x) + w_{t}(1 - \frac{1}{4R}, x),
\endaligned
\end{equation}
and let $v$ and $w$ solve the system of equations:
\begin{equation}\label{4.8}
\aligned
v_{tt} - \Delta v &= 0, \qquad \text{for} \qquad 1 - \frac{1}{4R} \leq t \leq 1, \\ v_{tt} - \Delta v + (|v + w|^{p - 1}(v + w) - |w|^{p - 1} w) &= 0, \qquad \text{for} \qquad t \geq 1,
\endaligned
\end{equation}
and
\begin{equation}\label{4.9}
\aligned
w_{tt} - \Delta w + |w + v|^{p - 1} (w + v) &= 0, \qquad \text{for} \qquad 1 - \frac{1}{4R} \leq t \leq 1, \\ w_{tt} - \Delta w + |w|^{p - 1} w &= 0, \qquad \text{for} \qquad t \geq 1.
\endaligned
\end{equation}
Using the small data arguments in \cite{lindblad1995existence} combined with $(\ref{2.13})$, $(\ref{4.3})$, and $(\ref{4.4})$,
\begin{equation}\label{4.10}
\| w \|_{L_{t,x}^{2(p - 1)}(\mathbb{R} \times \mathbb{R}^{3})} \lesssim \epsilon.
\end{equation}
Therefore, $(\ref{4.5})$ is equivalent to
\begin{equation}\label{4.11}
\| v \|_{L_{t,x}^{2(p - 1)}([1 - \frac{1}{4R}, \infty) \times \mathbb{R}^{3})} < \infty.
\end{equation}

The proof of $(\ref{4.11})$ will make use of some additional estimates on $w$.
\begin{lemma}\label{l4.1}
There exists a sequence $a_{j} \in l^{1}(\mathbb{Z})$ such that
\begin{equation}\label{4.12}
w(t,x) = \sum_{j \in \mathbb{Z}} w_{j}(t, x), \qquad \text{for any} \qquad t \in [1, \infty), \qquad x \in \mathbb{R}^{3},
\end{equation}
\begin{equation}\label{4.13}
|w_{j}(t,x)| \leq a_{j} 2^{-j \frac{p - 3}{p - 1}} (t - 1 + \frac{1}{4R})^{-1}, \qquad \text{for any} \qquad t \in [1, \infty), \qquad x \in \mathbb{R}^{3},
\end{equation}
\begin{equation}\label{4.14}
|\nabla w_{j}(t,x)| + |\partial_{t} w_{j}(t,x)| \leq a_{j} 2^{j \frac{2}{p - 1}} (t - 1 + \frac{1}{4R})^{-1}, \qquad \text{for any} \qquad t \in [1, \infty), \qquad x \in \mathbb{R}^{3},
\end{equation}
and
\begin{equation}\label{4.15}
|w_{j}(t,x)|  \leq a_{j} 2^{j \frac{2}{p - 1}}, \qquad \text{for any} \qquad t \in [1, \infty), \qquad x \in \mathbb{R}^{3},
\end{equation}
where
\begin{equation}\label{4.16}
\sum_{j} a_{j} \lesssim \epsilon.
\end{equation}
\end{lemma}
\begin{proof}
Using the scaling symmetry in $(\ref{1.2})$, Lemma $\ref{l4.1}$ is equivalent to proving the bounds in $(\ref{4.13})$--$(\ref{4.15})$ with $t - 1 + \frac{1}{4R}$ replaced by $t$, for $w$ solving
\begin{equation}\label{4.17}
\aligned
w_{tt} - \Delta w + |u|^{p - 1} u = 0, \qquad \text{for} \qquad 0 \leq t \leq 1, \\
w_{tt} - \Delta w + |w|^{p - 1} w = 0, \qquad \text{for} \qquad t \geq 1,
\endaligned
\end{equation}
with initial data satisfying $(\ref{2.9})$. First, by Theorem $\ref{t2.3}$ and the scaling symmetry $(\ref{1.2})$, $(\ref{4.13})$--$(\ref{4.16})$ hold for the Littlewood--Paley decomposition of
\begin{equation}\label{4.18}
\int_{0}^{1} S(t - \tau)(0, |u|^{p - 1} u) d\tau.
\end{equation}
Also, by the dispersive estimates in $(\ref{2.20})$, the bounds in $(\ref{4.13})$--$(\ref{4.15})$ also hold for $S(t)(P_{\geq j_{0}} u_{0}, P_{\geq j_{0}} u_{1})$. Meanwhile, by the dominated convergence theorem, for $R$ sufficiently large,
\begin{equation}\label{4.19}
\| (1 - \chi(\frac{x}{R})) P_{j} u_{0} \|_{L^{1}} \leq a_{j} 2^{-(\frac{3}{2} + s_{c})j},
\end{equation}
and
\begin{equation}\label{4.20}
\| \nabla^{3} (1 - \chi(\frac{x}{R})) P_{j} u_{0} \|_{L^{1}} \lesssim a_{j} 2^{(\frac{3}{2} - s_{c})j} + \frac{1}{R^{3}} \| \chi'''(\frac{x}{R}) P_{j} u_{0} \|_{L^{1}} \lesssim a_{j} 2^{(\frac{3}{2} - s_{c})j}.
\end{equation}
Similar computations also hold for $u_{1}$, so the bounds in $(\ref{4.13})$--$(\ref{4.15})$ hold for
\begin{equation}\label{4.21}
S(t)(P_{\leq j_{0}} u_{0}, P_{\leq j_{0}} u_{1}) + \int_{0}^{1} S(t - \tau)(0, |u|^{p - 1} u) d\tau.
\end{equation}
Combining the bounds for $(\ref{4.21})$ with $(\ref{4.10})$ and the proofs of Theorems $\ref{t2.2}$ and $\ref{t2.3}$ proves $(\ref{4.13})$--$(\ref{4.15})$.
\end{proof}

Returning to the solutions to $(\ref{4.8})$ and $(\ref{4.9})$ with initial data given by $(\ref{4.7})$, to prove $(\ref{4.11})$ we will use the conformal change of coordinates, similar to the computations in \cite{shen2017scattering}, \cite{dodson2018global} and \cite{dodson2018global2}. First observe that by the finite propagation speed, $v$ is supported on the set
\begin{equation}\label{4.22}
\{ (t, x) : t \geq 1, \qquad |x| \leq t - \frac{1}{2} + \frac{1}{400} \}.
\end{equation}
Since $\frac{3}{2} + \frac{1}{400} < \sqrt{3}$,
\begin{equation}\label{4.23}
\{ (t, x) : |x| \geq t - \frac{1}{2} + \frac{1}{400} \} \cap \{ (t, x) : t \geq 2 \} \subset \{ (t, x) : t^{2} - |x|^{2} \geq 1 \} \cap \{ (t, x) : t \geq 2 \}.
\end{equation}
In fact, there exists some $\delta_{0} > 0$ such that for any $0 \leq \delta \leq \delta_{0}$,
\begin{equation}\label{4.24}
\{ (t, x) : |x| \geq t - \frac{1}{2} + \frac{1}{400} \} \cap \{ (t, x) : t \geq 2 \} \subset \{ (t, x) : t^{2} - |x|^{2} \geq e^{2 \delta} \} \cap \{ (t, x) : t \geq 2 \}.
\end{equation}

Let
\begin{equation}\label{4.25}
\tilde{u}(\tau, y) = \frac{e^{\tau} \sinh |y|}{|y|} u(e^{\tau} \cosh |y|, e^{\tau} \sinh |y| \frac{y}{|y|}).
\end{equation}
By direct computation (see \cite{shen2017scattering}) for more information, if $u$ solves
\begin{equation}\label{4.26}
u_{tt} - \Delta u = F,
\end{equation}
inside $(\ref{4.23})$, then
\begin{equation}\label{4.27}
(\partial_{\tau \tau} - \Delta_{y}) \tilde{u} = e^{3 \tau} \frac{\sinh |y|}{|y|} F(e^{\tau} \cosh |y|, e^{\tau} \sinh |y| \frac{y}{|y|}),
\end{equation}
when $\tau > 0$. In particular, for
\begin{equation}\label{4.28}
\tilde{v}(\tau, y) = \frac{e^{\tau} \sinh |y|}{|y|} v(e^{\tau} \cosh |y|, e^{\tau} \sinh |y| \frac{y}{|y|}), \qquad \tilde{w}(\tau, y) = \frac{e^{\tau} \sinh |y|}{|y|} w(e^{\tau} \cosh |y|, e^{\tau} \sinh |y| \frac{y}{|y|}),
\end{equation}
\begin{equation}\label{4.29}
(\partial_{\tau \tau} - \Delta_{y}) \tilde{w} + e^{-(p - 3) \tau} (\frac{|y|}{\sinh |y|})^{p - 1} |\tilde{w}|^{p - 1} \tilde{w} = 0,
\end{equation}
and
\begin{equation}\label{4.30}
(\partial_{\tau \tau} - \Delta_{y}) \tilde{v} + e^{-(p - 3) \tau} (\frac{|y|}{\sinh |y|})^{p - 1} [|\tilde{u}|^{p - 1} \tilde{u} - |\tilde{w}|^{p - 1} \tilde{w}] = 0.
\end{equation}

Let $E(\tilde{v})$ denote the hyperbolic energy of $\tilde{v}$,
\begin{equation}\label{4.31}
E(\tilde{v}) = \frac{1}{2} \| \tilde{v}_{\tau} \|_{L^{2}}^{2} + \frac{1}{2} \| \nabla \tilde{v} \|_{L^{2}}^{2} + \frac{1}{p + 1} \int e^{-(p - 3) \tau} |\tilde{v}|^{p + 1} (\frac{|y|}{\sinh |y|})^{p - 1} dy.
\end{equation}

\begin{lemma}\label{l4.2}
There exists some $0 \leq \tau_{0} \leq \delta_{0}$ such that $(\ref{4.31})$ is finite.
\end{lemma}
\begin{proof}
To prove Lemma $\ref{l4.2}$, it suffices to prove
\begin{equation}\label{4.32}
\int_{0}^{\delta_{0}} \int \tilde{v}_{\tau}(\tau, y)^{2} + |\nabla \tilde{v}|^{2} + e^{-(p - 3) \tau} (\frac{|y|}{\sinh |y|})^{p - 1} |\tilde{v}|^{p + 1} dy d\tau < \infty.
\end{equation}
By direct computation,
\begin{equation}\label{4.33}
\aligned
\tilde{v}_{\tau}(\tau, y) = \frac{e^{2 \tau} \sinh |y| \cosh |y|}{|y|} v_{t}(e^{\tau} \cosh |y|, e^{\tau} \sinh |y| \frac{y}{|y|}) + \frac{e^{2 \tau} \sinh^{2} |y|}{|y|} v_{r}(e^{\tau} \cosh |y|, e^{\tau} \sinh |y| \frac{y}{|y|}) \\ + \frac{e^{\tau} \sinh |y|}{|y|} v(e^{\tau} \cosh |y|, e^{\tau} \sinh |y| \frac{y}{|y|}).
\endaligned
\end{equation}
By the support properties in $(\ref{4.23})$ and $(\ref{4.24})$ and the change of variables formula in \cite{shen2017scattering}, and the proof of Theorem $\ref{t1.6}$ in section three,
\begin{equation}\label{4.34}
\int_{0}^{\delta_{0}} \int \frac{e^{4 \tau} \sinh^{2} |y| \cosh^{2} |y|}{|y|^{2}} v_{t}^{2}(e^{\tau} \cosh |y|, e^{\tau} \sinh |y|)^{2} dy d\tau \lesssim \int_{1}^{2} \int v_{t}(t, y)^{2} dy dt \lesssim_{R} 1,
\end{equation}
\begin{equation}\label{4.35}
\int_{0}^{\delta_{0}} \int \frac{e^{4 \tau} \sinh^{4} |y|}{|y|^{2}} v_{r}^{2}(e^{\tau} \cosh |y|, e^{\tau} \sinh |y|)^{2} dy d\tau \lesssim \int_{1}^{2} \int v_{r}(t, y)^{2} dy dt \lesssim_{R} 1,
\end{equation}
and by Hardy's inequality,
\begin{equation}\label{4.36}
\int_{0}^{\delta_{0}} \int \frac{e^{2\tau} \sinh^{2} |y|}{|y|^{2}} v^{2}(e^{\tau} \cosh |y|, e^{\tau} \sinh |y| \frac{y}{|y|}) \lesssim \int_{1}^{2} \int \frac{1}{|y|^{2}} v(t, y)^{2} dy dt \lesssim_{R} 1.
\end{equation}
Therefore, $\int_{0}^{\delta_{0}} \int \tilde{v}_{\tau}(\tau, y)^{2} dy d\tau \lesssim_{R} 1$. A similar computation proves $\int_{0}^{\delta_{0}} \int |\nabla \tilde{v}(\tau, y)|^{2} dy d\tau \lesssim_{R} 1.$

Finally,
\begin{equation}\label{4.37}
\aligned
\int_{0}^{\delta_{0}} \int \frac{e^{(p + 1)\tau} \sinh^{p + 1} |y|}{|y|^{p + 1}} |v(e^{\tau} \cosh |y|, e^{\tau} \sinh |y| \frac{y}{|y|})|^{p + 1} e^{-(p - 3) \tau} (\frac{|y|}{\sinh |y|})^{p - 1} dy d\tau \\
\lesssim \int_{1}^{2} \int |v(t, y)|^{p + 1} dy dt \lesssim_{R} 1.
\endaligned
\end{equation}
This proves $(\ref{4.32})$, which proves the Lemma.
\end{proof}

The finite energy at $\tau_{0}$ grows very slowly.
\begin{theorem}\label{t4.3}
For $\tau_{0} \leq \tau \leq 1$,
\begin{equation}\label{4.38}
E(\tilde{v}) \lesssim_{R} E(\tilde{v}(\delta_{0})).
\end{equation}
\end{theorem}
\begin{proof}
Computing the change of the hyperbolic energy,
\begin{equation}\label{4.39}
\aligned
\frac{d}{d\tau} E(\tilde{v}) = -\frac{p - 3}{p - 1} \int e^{-(p - 3) \tau} |\tilde{v}|^{p + 1} (\frac{|y|}{\sinh |y|})^{p - 1} dy \\ - \frac{1}{p + 1} \int e^{-(p - 3) \tau} (\frac{|y|}{\sinh |y|})^{p - 1} \tilde{v}_{\tau} [|\tilde{u}|^{p - 1} \tilde{u} - |\tilde{v}|^{p - 1} \tilde{v} - |\tilde{w}|^{p - 1} \tilde{w}] dy.
\endaligned
\end{equation}
By $(\ref{3.6})$, for $3 < p < 5$,
\begin{equation}\label{4.40}
|\tilde{u}|^{p - 1} \tilde{u} - |\tilde{v}|^{p - 1} \tilde{v} - |\tilde{w}|^{p - 1} \tilde{w} =  p |\tilde{v}|^{p - 1} \tilde{w} + O(|\tilde{w}|^{\frac{p - 1}{2}} |\tilde{v}|^{\frac{p + 1}{2}}) + O(|\tilde{w}|^{p - 1} |\tilde{v}|).
\end{equation}
By H{\"o}lder's inequality,
\begin{equation}\label{4.41}
\aligned
\int e^{-(p - 3) \tau} (\frac{|y|}{\sinh |y|})^{p - 1} |\tilde{w}|^{\frac{p - 1}{2}} |\tilde{v}|^{\frac{p + 1}{2}} |\tilde{v}_{\tau}| dy \\
\lesssim \| e^{\tau} |w(e^{\tau} \cosh |y|, e^{\tau} \sinh |y| \frac{y}{|y|})|^{\frac{p - 1}{2}} \|_{L^{\infty}} (\int e^{-(p - 3) \tau} (\frac{|y|}{\sinh |y|})^{p - 1} |\tilde{v}|^{p + 1} dy)^{1/2} \| \tilde{v}_{\tau} \|_{L^{2}} \\
\lesssim E(\tilde{v}) \| e^{\tau} |w(e^{\tau} \cosh |y|, e^{\tau} \sinh |y| \frac{y}{|y|})|^{\frac{p - 1}{2}} \|_{L^{\infty}}.
\endaligned
\end{equation}

By Hardy's inequality and the Sobolev embedding theorem,
\begin{equation}\label{4.42}
\aligned
\int e^{-(p - 3) \tau} (\frac{|y|}{\sinh |y|})^{p - 1} |\tilde{v}_{\tau}| |\tilde{v}| |\tilde{w}|^{p - 1} dy \lesssim \| \tilde{v}_{\tau} \|_{L^{2}} \| \frac{1}{|y|} \tilde{v} \|_{L^{2}} \| e^{\tau} |y| |w(e^{\tau} \cosh |y|, e^{\tau} \sinh |y| \frac{y}{|y|})|^{p - 1} \|_{L^{\infty}}
 \\ \lesssim E(\tilde{v}) \| e^{\tau} |w(e^{\tau} \cosh |y|, e^{\tau} \sinh |y| \frac{y}{|y|})|^{\frac{p - 1}{2}} \|_{L^{\infty}} \| |y| |w(e^{\tau} \cosh |y|, e^{\tau} \sinh |y| \frac{y}{|y|})|^{\frac{p - 1}{2}} \|_{L^{\infty}}.
 \endaligned
\end{equation}
By Lemma $\ref{l4.1}$, for any $j \in \mathbb{Z}$,
\begin{equation}\label{4.43}
|w_{j}(e^{\tau} \cosh |y|, e^{\tau} \sinh |y| \frac{y}{|y|})| \lesssim 2^{-j(\frac{p - 3}{p - 1})} (e^{\tau} \cosh |y| - 1 + \frac{1}{4R})^{-1} \lesssim 2^{-j \frac{p - 3}{p - 1}} (e^{\tau} - 1 + \frac{1}{4R})^{-1} a_{j},
\end{equation}
\begin{equation}\label{4.44}
|w_{j}(e^{\tau} \cosh |y|, e^{\tau} \sinh |y| \frac{y}{|y|})| \lesssim 2^{j \frac{2}{p - 1}} a_{j},
\end{equation}
When $j \geq 0$,
\begin{equation}\label{4.45}
\int_{0}^{2^{-j}} 2^{j} e^{\tau} d\tau \lesssim 1,
\end{equation}
and
\begin{equation}\label{4.46}
\int_{2^{-j}}^{\infty} e^{\tau} 2^{-j \frac{p - 3}{2}} (e^{\tau} - 1 + \frac{1}{4R})^{-\frac{p - 1}{2}} d\tau \lesssim \int_{2^{-j}}^{1} \tau^{-\frac{p - 1}{2}} 2^{-j \frac{p - 3}{2}} d\tau + 2^{-j \frac{p - 3}{2}} \int_{1}^{\infty} e^{-\frac{p - 3}{2} \tau} d\tau \lesssim_{p - 3} 1.
\end{equation}
When $j \leq 0$,
\begin{equation}\label{4.47}
\int_{0}^{1} 2^{j} e^{\tau} d\tau \lesssim 1,
\end{equation}
\begin{equation}\label{4.48}
\int_{1}^{-\ln(2) j} 2^{j} e^{\tau} d\tau \lesssim 1,
\end{equation}
and
\begin{equation}\label{4.49}
\int_{-\ln(2) j}^{\infty} e^{-j(\frac{p - 3}{2})} e^{-\tau \frac{p - 1}{2}} d\tau \lesssim 1.
\end{equation}
Therefore,
\begin{equation}\label{4.50}
\int_{0}^{\infty} \| e^{\tau} |w(e^{\tau} \cosh |y|, e^{\tau} \sinh |y| \frac{y}{|y|})|^{\frac{p - 1}{2}} \|_{L^{\infty}} d\tau \lesssim \epsilon.
\end{equation}
Also, by Lemma $\ref{l4.1}$,
\begin{equation}\label{4.51}
\| |y| |w(e^{\tau} \cosh |y|, e^{\tau} \sinh |y| \frac{y}{|y|})|^{\frac{p - 1}{2}} \|_{L^{\infty}} \lesssim \epsilon.
\end{equation}
Therefore, the contribution of $(\ref{4.41})$ and $(\ref{4.42})$ to $(\ref{4.39})$ may be absorbed into the left hand side of $(\ref{4.39})$, proving
\begin{equation}\label{4.52}
\sup_{\tau_{0} \leq \tau \leq T_{0}} E(\tilde{v}(\tau)) \lesssim E(\tilde{v}(\delta_{0})) + \sup_{\tau_{0} \leq T \leq T_{0}} \int_{\tau_{0}}^{T} \int p |\tilde{v}|^{p - 1} \tilde{v}_{\tau} \tilde{w} (\frac{|y|}{\sinh |y|})^{p - 1} e^{-(p - 3) \tau} dy d\tau.
\end{equation}
Lemma $\ref{l4.1}$ implies that to bound the second term on the right hand side of $(\ref{4.52})$, it suffices to obtain a bound of
\begin{equation}\label{4.52.1}
\sup_{\tau_{0} \leq T \leq 1} \int_{\tau_{0}}^{T} \int p |\tilde{v}|^{p - 1} \tilde{v}_{\tau} \tilde{w}_{j} (\frac{|y|}{\sinh |y|})^{p - 1} e^{-(p - 3) \tau} dy d\tau,
\end{equation}
with a bound summable in $j$.

To prove this bound we will use a modification of the Littlewood--Paley decomposition. Let $\psi \in C_{0}^{\infty}(\mathbb{R})$ be a smooth function satisfying $\psi(x) \geq 0$ on $\mathbb{R}$, $\int \psi(x) dx = 1$, and $\psi(x)$ is supported on $|x| \leq 1$. Then for $f \in L_{\tau}^{1}$, set
\begin{equation}\label{4.53}
\aligned
\tilde{P}_{0} f = \int \psi(\tau - s) f(s), \qquad \text{and for} \qquad k > 0, \\ \qquad \tilde{P}_{k} f = \int 2^{k} \psi(2^{k}(\tau - s)) f(s) ds - \int 2^{k - 1} \psi(2^{k - 1}(\tau - s)) f(s) ds.
\endaligned
\end{equation}
Also observe that for any $k > 0$, summing up the telescoping sum in $(\ref{4.53})$,
\begin{equation}\label{4.53.1}
\aligned
 \tilde{P}_{\leq k} f = \int 2^{k} \psi(2^{k}(\tau - s)) f(s) ds.
\endaligned
\end{equation}

Suppose $E(\tilde{v})$ is bounded on the interval $[\tau_{0}, T]$. Then by local well-posedness arguments,
\begin{equation}\label{4.54}
1_{[\tau_{0}, T]} \partial_{\tau}(|\tilde{v}|^{p - 1} \tilde{v}) \in L_{\tau, y}^{1},
\end{equation}
so
\begin{equation}\label{4.55}
\aligned
\tilde{P}_{0}(1_{[\tau_{0}, T]} \partial_{\tau}(|\tilde{v}|^{p - 1} \tilde{v})) + \sum_{j \geq 1} \tilde{P}_{j}(1_{[\tau_{0}, T]} \partial_{\tau}(|\tilde{v}|^{p - 1} \tilde{v})) \\ = \partial_{\tau} \tilde{P}_{0}(1_{[\tau_{0}, T]} (|\tilde{v}|^{p - 1} \tilde{v})) + \sum_{j \geq 1} \tilde{P}_{j}\partial_{\tau} (1_{[\tau_{0}, T]} (|\tilde{v}|^{p - 1} \tilde{v})) - (|\tilde{v}|^{p - 1} \tilde{v})|_{\tau_{0}}^{T}.
\endaligned
\end{equation}
for almost every $\tau \in \mathbb{R}$, where $1_{[a, b]}$ is the indicator function of the interval $[a, b]$.

The second term on the right hand side of $(\ref{4.55})$ can be computed using Hardy's inequality and Lemma $\ref{l4.1}$,
\begin{equation}\label{4.57}
\aligned
\int (|\tilde{v}|^{p - 1} \tilde{v}) (\frac{|y|}{|\sinh |y|})^{p - 2} e^{-(p - 3) \frac{p - 2}{p - 1} \tau} [e^{\frac{2}{p - 1} \tau} w_{j}(e^{\tau} \cosh |y|, e^{\tau} \sinh |y| \frac{y}{|y|})]|_{\tau_{0}}^{T} dy \\
\lesssim (\int e^{-(p - 3) \tau} (\frac{|y|}{\sinh |y|})^{p - 1} |\tilde{v}|^{p + 1} dy)^{\frac{p - 2}{p - 1}} \| \frac{1}{|y|} \tilde{v} \|_{L^{2}}^{\frac{2}{p - 1}} \| |y|^{\frac{2}{p - 1}} e^{\frac{2}{p - 1} \tau} |w_{j}(e^{\tau} \cosh |y|, e^{\tau} \sinh |y| \frac{y}{|y|})| \|_{L^{\infty}}|_{\tau_{0}}^{T} \\
\lesssim E(\tilde{v}) \| |y|^{\frac{2}{p - 1}} e^{\frac{2}{p - 1} \tau} |w_{j}(e^{\tau} \cosh |y|, e^{\tau} \sinh |y| \frac{y}{|y|})| \|_{L^{\infty}}|_{\tau_{0}}^{T} \lesssim a_{j} \sup_{\tau \in [\tau_{0}, T]} E(\tilde{v}).
\endaligned
\end{equation}
The sum of these terms in $j$ can be absorbed into the left hand side of $(\ref{4.52})$.

To handle the first term on the right hand side of $(\ref{4.55})$, it is useful to consider a number of cases separately.\medskip

\noindent \textbf{Case 1, $2^{j} e^{|y|} \leq 1$:} By Lemma $\ref{l4.1}$,
\begin{equation}\label{4.56}
\aligned
 \int_{\tau_{0}}^{T} \partial_{\tau} \tilde{P}_{0}(1_{[\delta_{0}, T]} (|\tilde{v}|^{p - 1} \tilde{v})) (\frac{|y|}{|\sinh |y|})^{p - 2} e^{-(p - 3) \frac{p - 2}{p - 1} \tau} [e^{\frac{2}{p - 1} \tau} w_{j}(e^{\tau} \cosh |y|, e^{\tau} \sinh |y| \frac{y}{|y|})] d\tau \\
\lesssim \int_{\tau_{0}}^{T} |\frac{1}{|y|} \tilde{v}|^{\frac{2}{p - 1}} (e^{-(p - 3) \tau} (\frac{|y|}{\sinh |y|})^{p - 1} |\tilde{v}|^{p + 1})^{\frac{p - 2}{p - 1}} d\tau \cdot \sup_{\tau \in [\tau_{0}, T]} (|y|^{\frac{2}{p - 1}} e^{\tau \frac{2}{p - 1}}  |w_{j}(e^{\tau} \cosh |y|, e^{\tau} \sinh |y| \frac{y}{|y|})|) \\
\lesssim a_{j}  \int_{\tau_{0}}^{T} |\frac{1}{|y|} \tilde{v}|^{\frac{2}{p - 1}} (e^{-(p - 3) \tau} (\frac{|y|}{\sinh |y|})^{p - 1} |\tilde{v}|^{p + 1})^{\frac{p - 2}{p - 1}} d\tau.
\endaligned
\end{equation}
Next, by the fundamental theorem of calculus, for any $k \geq 0$,
\begin{equation}\label{4.62}
\tilde{P}_{> k} f(\tau) = f(\tau) - 2^{k} \int \psi(2^{k}(\tau - s)) f(s) ds = 2^{k} \int \psi(2^{k}(\tau - s))[f(\tau) - f(s)] ds = 2^{k} \int \psi(2^{k}(\tau - s)) \int_{s}^{\tau} f'(r) dr.
\end{equation}
Integrating by parts and following $(\ref{4.57})$ for the third term, $(\ref{4.62})$ with $j = 0$ for the second, and $(\ref{4.56})$ for the first,
\begin{equation}\label{4.60}
\int_{\tau_{0}}^{T} \partial_{\tau} \tilde{P}_{> 0}(1_{[\delta_{0}, T]} |\tilde{v}|^{p - 1} \tilde{v}) (\frac{|y|}{\sinh |y|})^{p - 2} e^{-(p - 3) \frac{p - 2}{p - 1} \tau} [e^{\frac{2}{p - 1} \tau} w_{j}(e^{\tau} \cosh |y|, e^{\tau} \sinh |y| \frac{y}{|y|})] d\tau
\end{equation}
\begin{equation}\label{4.61}
\aligned
= (p - 4) \int_{\tau_{0}}^{T} \tilde{P}_{> 0}(1_{[\delta_{0}, T]} |\tilde{v}|^{p - 1} \tilde{v}) (\frac{|y|}{\sinh |y|})^{p - 2} e^{-(p - 3) \frac{p - 2}{p - 1} \tau} [e^{\frac{2}{p - 1} \tau} w_{j}(e^{\tau} \cosh |y|, e^{\tau} \sinh |y| \frac{y}{|y|})] d\tau \\
- \int_{\tau_{0}}^{T} \tilde{P}_{> 0}(1_{[\tau_{0}, T]} |\tilde{v}|^{p - 1} \tilde{v}) (\frac{|y|}{\sinh |y|})^{p - 2} e^{-(p - 3) \frac{p - 2}{p - 1} \tau}  e^{\frac{2}{p - 1} \tau} \\ \times [e^{\tau} \cosh |y| (\partial_{t} w_{j}) + e^{\tau} \sinh |y| (\partial_{r} w_{j})](e^{\tau} \cosh |y|, e^{\tau} \sinh |y| \frac{y}{|y|})] d\tau \\
+ \tilde{P}_{> 0}(1_{[\tau_{0}, T]} |\tilde{v}|^{p - 1} \tilde{v}) (\frac{|y|}{\sinh |y|})^{p - 2} e^{-(p - 3) \frac{p - 2}{p - 1} \tau} [e^{\frac{2}{p - 1} \tau} w_{j}(e^{\tau} \cosh |y|, e^{\tau} \sinh |y| \frac{y}{|y|})]|_{\tau_{0}}^{T}
\endaligned
\end{equation}
\begin{equation}\label{4.59}
\aligned
\lesssim a_{j}  \int_{\tau_{0}}^{T} |\frac{1}{|y|} \tilde{v}|^{\frac{2}{p - 1}} (e^{-(p - 3) \tau} (\frac{|y|}{\sinh |y|})^{p - 1} |\tilde{v}|^{p + 1})^{\frac{p - 2}{p - 1}} d\tau \\ + a_{j}  \int_{\tau_{0}}^{T} |\tilde{v}_{\tau}|^{\frac{2}{p - 1}} (e^{-(p - 3) \tau} (\frac{|y|}{\sinh |y|})^{p - 1} |\tilde{v}|^{p + 1})^{\frac{p - 2}{p - 1}} d\tau + a_{j} \sup_{\tau \in [\tau_{0}, T]} E(\tilde{v}).
\endaligned
\end{equation}
These estimates are acceptable for our purposes.\medskip

\noindent \textbf{Case 2, $2^{j} e^{|y|} \sim 2^{k} \geq 1$:} In this case, the contribution of $\tilde{P}_{> k}$ may be handled in a manner very similar to $(\ref{4.60})$ and $(\ref{4.61})$. Indeed,
\begin{equation}\label{4.63}
\int_{\tau_{0}}^{T} \partial_{\tau} \tilde{P}_{> k}(1_{[\delta_{0}, T]} |\tilde{v}|^{p - 1} \tilde{v}) (\frac{|y|}{\sinh |y|})^{p - 2} e^{-(p - 3) \frac{p - 2}{p - 1} \tau} [e^{\frac{2}{p - 1} \tau} w_{j}(e^{\tau} \cosh |y|, e^{\tau} \sinh |y| \frac{y}{|y|})] d\tau
\end{equation}
\begin{equation}\label{4.64}
\aligned
= (p - 4) \int_{\tau_{0}}^{T} \tilde{P}_{> k}(1_{[\delta_{0}, T]} |\tilde{v}|^{p - 1} \tilde{v}) (\frac{|y|}{\sinh |y|})^{p - 2} e^{-(p - 3) \frac{p - 2}{p - 1} \tau} [e^{\frac{2}{p - 1} \tau} w_{j}(e^{\tau} \cosh |y|, e^{\tau} \sinh |y| \frac{y}{|y|})] d\tau \\
- \int_{\tau_{0}}^{T} \tilde{P}_{> k}(1_{[\delta_{0}, T]} |\tilde{v}|^{p - 1} \tilde{v}) (\frac{|y|}{\sinh |y|})^{p - 2} e^{-(p - 3) \frac{p - 2}{p - 1} \tau}  e^{\frac{2}{p - 1} \tau} \\ \times [e^{\tau} \cosh |y| (\partial_{t} w_{j}) + e^{\tau} \sinh |y| (\partial_{r} w_{j})](e^{\tau} \cosh |y|, e^{\tau} \sinh |y| \frac{y}{|y|})] d\tau \\
+ \tilde{P}_{> k}(1_{[\tau_{0}, T]} |\tilde{v}|^{p - 1} \tilde{v}) (\frac{|y|}{\sinh |y|})^{p - 2} e^{-(p - 3) \frac{p - 2}{p - 1} \tau} [e^{\frac{2}{p - 1} \tau} w_{j}(e^{\tau} \cosh |y|, e^{\tau} \sinh |y| \frac{y}{|y|})]|_{\tau_{0}}^{T}
\endaligned
\end{equation}
Following the computations in $(\ref{4.62})$ and using the fact that Lemma $\ref{l4.1}$ implies
\begin{equation}\label{4.64.1}
\|  [e^{\tau} \cosh |y| (\partial_{t} w_{j}) + e^{\tau} \sinh |y| (\partial_{r} w_{j})](e^{\tau} \cosh |y|, e^{\tau} \sinh |y| \frac{y}{|y|})] \|_{L^{\infty}} \lesssim R e^{k \frac{2}{p - 1}}.
\end{equation}
The computations in $(\ref{4.56})$ may be copied over in this case. Finally, take $l \in \mathbb{Z}$, $0 < l \leq k$. In this case, by Lemma $\ref{l4.1}$,
\begin{equation}\label{4.65}
\aligned
\int_{\tau_{0}}^{T} \partial_{\tau} \tilde{P}_{l}(1_{[\delta_{0}, T]} |\tilde{v}|^{p - 1} \tilde{v}) (\frac{|y|}{\sinh |y|})^{p - 2} e^{-(p - 3) \frac{p - 2}{p - 1} \tau} [e^{\frac{2}{p - 1} \tau} w_{j}(e^{\tau} \cosh |y|, e^{\tau} \sinh |y| \frac{y}{|y|})] d\tau \\
\lesssim 2^{\frac{p - 3}{p - 1}(l - k)} a_{j} R \int_{\tau_{0}}^{T} |\tilde{v}_{\tau}|^{\frac{2}{p - 1}} (e^{-(p - 3) \tau} (\frac{|y|}{\sinh |y|})^{p - 1} |\tilde{v}|^{p + 1})^{\frac{p - 2}{p - 1}} d\tau.
\endaligned
\end{equation}
Summing up in $j$ and $l$, and integrating in $y$, we have therefore proved
\begin{equation}\label{4.66}
\sup_{\tau_{0} \leq \tau \leq T_{0}} E(\tilde{v}(\tau)) \lesssim E(\tilde{v}(\delta_{0})) + \epsilon R \int_{\tau_{0}}^{T_{0}} (\int \frac{1}{|y|^{2}} |\tilde{v}|^{2} + |\tilde{v}_{\tau}|^{2} dy)^{\frac{1}{p - 1}} \cdot (e^{-(p - 3) \tau} (\frac{|y|}{\sinh |y|})^{p - 1} |\tilde{v}|^{p + 1} dy)^{\frac{p - 2}{p - 1}} d\tau.
\end{equation}
Taking $T_{0} = \tau_{0} + \frac{1}{R}$, and making a standard bootstrap argument, it is possible to absorb the second term on the right hand side of $(\ref{4.66})$ into the left hand side. Iterating this argument $O_{R}(1)$ times proves the Theorem.
\end{proof}

We can upgrade this to a global integral result.
\begin{theorem}\label{t4.4}
\begin{equation}\label{4.66.1}
\int_{1}^{\infty} \int |\tilde{v}|^{p + 1} e^{-(p - 3) \tau} (\frac{|y|}{\sinh |y|})^{p - 1} dy d\tau \lesssim_{R} 1.
\end{equation}
\end{theorem}
\begin{proof}
When $\tau \geq 1$, by Lemma $\ref{l4.1}$,
\begin{equation}\label{4.67}
e^{\tau \frac{2}{p - 1}} |w_{j}(e^{\tau} \cosh |y|, e^{\tau} \sinh |y| \frac{y}{|y|})| \lesssim e^{-\tau \frac{p - 3}{p - 1}} 2^{-j \frac{p - 3}{p - 1}} \cosh(|y|)^{-1},
\end{equation}
and
\begin{equation}\label{4.68}
\aligned
e^{\tau \frac{2}{p - 1}} e^{\tau} \cosh |y| |(\partial_{t} w_{j})(e^{\tau} \cosh |y|, e^{\tau} \sinh |y| \frac{y}{|y|})| \\ + e^{\tau \frac{2}{p - 1}} e^{\tau} \cosh |y| |(\partial_{r} w_{j})(e^{\tau} \cosh |y|, e^{\tau} \sinh |y| \frac{y}{|y|})| \lesssim e^{\tau \frac{2}{p - 1}} 2^{j \frac{2}{p - 1}}.
\endaligned
\end{equation}

Revisiting $(\ref{4.39})$ and $(\ref{4.52})$,
\begin{equation}\label{4.69}
\aligned
\int_{1}^{T_{0}} \int e^{-(p - 3) \tau} (\frac{|y|}{\sinh |y|})^{p - 1} |\tilde{v}|^{p + 1} dy d\tau + \sup_{1 \leq \tau \leq T_{0}} E(\tilde{v}(\tau)) \lesssim E(\tilde{v}(1)) \\ + \sup_{1 \leq T \leq T_{0}} \int_{1}^{T} \int p |\tilde{v}|^{p - 1} \tilde{v}_{\tau} \tilde{w} (\frac{|y|}{\sinh |y|})^{p - 1} e^{-(p - 3) \tau} dy d\tau.
\endaligned
\end{equation}
Now take the partition of unity
\begin{equation}\label{4.70}
1 = \sum_{m \in \mathbb{Z}} \chi(\tau - m), \qquad \text{which satisfies} \qquad \sum_{m \in \mathbb{Z}} |\chi'(\tau - m)| \lesssim 1.
\end{equation}
Let $k(m, y) = \sup \{ 0, j + \frac{|y|}{\ln(2)} + \frac{m}{\ln(2)} \}$. Integrating by parts in $\tau$, as in $(\ref{4.60})$ and $(\ref{4.61})$,
\begin{equation}\label{4.71}
\sum_{m} \int_{1}^{T} \partial_{\tau} \tilde{P}_{> k(m, y)}(1_{[1, T]} |\tilde{v}|^{p - 1} \tilde{v}) (\frac{|y|}{\sinh |y|})^{p - 2} e^{-(p - 3) \frac{p - 2}{p - 1} \tau} \cdot \chi(\tau - m) [e^{\frac{2}{p - 1} \tau} w_{j}(e^{\tau} \cosh |y|, e^{\tau} \sinh |y| \frac{y}{|y|})] d\tau
\end{equation}
\begin{equation}\label{4.71}
\aligned
\lesssim \sum_{m} \int_{1}^{T} \tilde{P}_{> k(m, y)}(1_{[1, T]} |\tilde{v}|^{p - 1} \tilde{v}) (\frac{|y|}{\sinh |y|})^{p - 2} e^{-(p - 3) \frac{p - 2}{p - 1} \tau} \\ \times (|\chi(\tau - m)| + |\chi'(\tau - m)|) [e^{\frac{2}{p - 1} \tau} w_{j}(e^{\tau} \cosh |y|, e^{\tau} \sinh |y| \frac{y}{|y|})] d\tau \\
- \sum_{m} \int_{1}^{T} \tilde{P}_{> k(m, y)}(1_{[1, T]} |\tilde{v}|^{p - 1} \tilde{v}) (\frac{|y|}{\sinh |y|})^{p - 2} e^{-(p - 3) \frac{p - 2}{p - 1} \tau}  e^{\frac{2}{p - 1} \tau} \\ \times \chi(\tau - m) [e^{\tau} \cosh |y| (\partial_{t} w_{j}) + e^{\tau} \sinh |y| (\partial_{r} w_{j})](e^{\tau} \cosh |y|, e^{\tau} \sinh |y| \frac{y}{|y|})] d\tau \\
+ \sum_{m} \tilde{P}_{> k(m, y)} \chi(\tau - m) (1_{[1, T]} |\tilde{v}|^{p - 1} \tilde{v}) (\frac{|y|}{\sinh |y|})^{p - 2} e^{-(p - 3) \frac{p - 2}{p - 1} \tau} [e^{\frac{2}{p - 1} \tau} w_{j}(e^{\tau} \cosh |y|, e^{\tau} \sinh |y| \frac{y}{|y|})]|_{1}^{T}.
\endaligned
\end{equation}
Therefore, using the computations leading up to $(\ref{4.66})$, by $(\ref{4.67})$,
\begin{equation}\label{4.72}
\int (\ref{4.71}) dy \lesssim a_{j} \int_{1}^{T} (\int \frac{1}{|y|^{2} \cosh^{2} |y|} |\tilde{v}|^{2} + \frac{1}{\cosh^{2} |y|} |\tilde{v}_{\tau}|^{2} dy)^{\frac{1}{p - 1}} \cdot (e^{-(p - 3) \tau} (\frac{|y|}{\sinh |y|})^{p - 1} |\tilde{v}|^{p + 1} dy)^{\frac{p - 2}{p - 1}} d\tau.
\end{equation}
Meanwhile, using the computations in $(\ref{4.56})$ and $(\ref{4.65})$,
\begin{equation}\label{4.73}
\aligned
\int \sum_{m} \int_{1}^{T} \partial_{\tau} \tilde{P}_{\leq k(m, y)}(1_{[1, T]} |\tilde{v}|^{p - 1} \tilde{v}) (\frac{|y|}{\sinh |y|})^{p - 2} e^{-(p - 3) \frac{p - 2}{p - 1} \tau} \\ \times \chi(\tau - m) [e^{\frac{2}{p - 1} \tau} w_{j}(e^{\tau} \cosh |y|, e^{\tau} \sinh |y| \frac{y}{|y|})] d\tau dy \\
\lesssim a_{j} \int_{1}^{T} (\int \frac{1}{|y|^{2} \cosh^{2} |y|} |\tilde{v}|^{2} + \frac{1}{\cosh^{2} |y|} |\tilde{v}_{\tau}|^{2} dy)^{\frac{1}{p - 1}} \cdot (e^{-(p - 3) \tau} (\frac{|y|}{\sinh |y|})^{p - 1} |\tilde{v}|^{p + 1} dy)^{\frac{p - 2}{p - 1}} d\tau.
\endaligned
\end{equation}
Summing in $j$,
\begin{equation}\label{4.74}
\aligned
\sum_{j} (\ref{4.72}) + (\ref{4.73}) \lesssim \epsilon (\int_{1}^{T} \int e^{-(p - 3) \tau} (\frac{|y|}{\sinh |y|})^{p - 1} |\tilde{v}|^{p + 1} dy d\tau) \\ + \epsilon  (\int_{1}^{T} \int \frac{1}{|y|^{2} \cosh^{2} |y|} |\tilde{v}|^{2} + \frac{1}{\cosh^{2} |y|} |\tilde{v}_{\tau}|^{2} dy d\tau).
\endaligned
\end{equation}
The first term on the right hand side of $(\ref{4.74})$ may be absorbed into the left hand side of $(\ref{4.69})$. The second term on the right hand side of $(\ref{4.74})$ can be controlled by a local energy decay estimate.

\begin{theorem}[Local energy decay]\label{t4.5}
\begin{equation}\label{4.75}
\aligned
\int_{1}^{T} \int \frac{1}{(1 + |y|^{2})^{3/2}} [\tilde{v}_{\tau}^{2} + |\nabla \tilde{v}|^{2}] dy d\tau + \int_{1}^{T} \int \frac{1}{(1 + |y|^{2})^{5/2}} \tilde{v}^{2} dy d\tau \\ \lesssim \sup_{\tau \in [1, T]} E(\tilde{v}) + \epsilon \int_{1}^{T} \int e^{-(p - 3) \tau} (\frac{|y|}{\sinh |y|})^{p - 1} |\tilde{v}|^{p + 1} dy d\tau.
\endaligned
\end{equation}
\end{theorem}
Postponing the proof of Theorem $\ref{t4.5}$, Theorem $\ref{t4.4}$ follows.
\end{proof}
The bounds in Theorem $\ref{t4.4}$ imply bounds on $\| v \|_{L_{t,x}^{2(p - 1)}([2, \infty) \times \mathbb{R}^{3})} < \infty$. Since $E(\tilde{v})$ is uniformly bounded, and $(\ref{4.66.1})$ is finite, partition $[0, \infty)$ into finitely many subintervals such that
\begin{equation}\label{4.76}
\int_{I_{j}} \int e^{-(p - 3) \tau} (\frac{|y|}{\sinh |y|})^{p - 1} |\tilde{v}(\tau, y)|^{p + 1} dy d\tau < \epsilon.
\end{equation}
For any $3 < p < 5$, by $(\ref{4.42})$ and $(\ref{4.30})$ there exists $\theta(p)$ such that for $\epsilon > 0$ sufficiently small,
\begin{equation}\label{4.77}
\| \tilde{v} \|_{\dot{S}^{1}(I_{j} \times \mathbb{R}^{3})} \lesssim \sup_{\tau} E(\tilde{v})^{1/2} + \epsilon^{\theta(p - 1)} \| \tilde{v} \|_{\dot{S}^{1}(I_{j} \times \mathbb{R}^{3})}^{1 + (1 - \theta)(p - 1)} + \| \tilde{v} \tilde{w}^{p - 1} e^{-(p - 3) \tau} (\frac{|y|}{\sinh |y|})^{p - 1} \|_{L_{\tau}^{1} L_{y}^{2}} \lesssim \sup_{\tau} E(\tilde{v})^{1/2}.
\end{equation}
Therefore,
\begin{equation}\label{4.78}
\| \tilde{v} \|_{L_{\tau, y}^{8}([\delta_{0}, \infty) \times \mathbb{R}^{3})} < \infty.
\end{equation}
Interpolating this bound with $(\ref{4.38})$ and $(\ref{4.66.1})$ then implies
\begin{equation}\label{4.79}
\int_{\tau_{0}}^{\infty} \int e^{-(p - 3) \tau} (\frac{|y|}{\sinh |y|})^{p - 1} |\tilde{v}(\tau, y)|^{2(p - 1)} dy d\tau < \infty.
\end{equation}
Using the change of variables formula, since $p - 1 > 2$,
\begin{equation}\label{4.80}
\aligned
\int_{\tau_{0}}^{\infty} \int e^{-(p - 3) \tau} (\frac{|y|}{\sinh |y|})^{p - 1} |\tilde{v}(\tau, y)|^{2(p - 1)} dy d\tau \\
= \int_{\tau_{0}}^{\infty} \int e^{2 \tau} |v(e^{\tau} \cosh |y|, e^{\tau} \sinh |y| \frac{y}{|y|})|^{p - 1} (\frac{e^{\tau} \sinh |y|}{|y|})^{p - 1} |v(e^{\tau} \cosh |y|, e^{\tau} \sinh |y| \frac{y}{|y|})|^{p - 1} dy d\tau \\
\geq \int_{t^{2} - |x|^{2} \geq e^{2 \tau_{0}}} |v(t, x)|^{2(p - 1)} dx dt.
\endaligned
\end{equation}
Since $v$ is supported in $(\ref{4.24})$ with $\delta = \tau_{0}$ and $t \geq 2$,
\begin{equation}\label{4.81}
\| v \|_{L_{t,x}^{2(p - 1)}([2, \infty) \times \mathbb{R}^{3})} < \infty.
\end{equation}
The global well-posedness results of the previous section combined with $(\ref{4.81})$ implies $(\ref{4.11})$.

\section{Local energy decay}
Theorem $\ref{t4.5}$ is proved using a virial identity. Let
\begin{equation}\label{5.1}
M(\tau) = \int \frac{y}{(1 + |y|^{2})^{1/2}} \tilde{v}_{\tau} \cdot \nabla \tilde{v} dy + \int \frac{1}{(1 + |y|^{2})^{1/2}} \tilde{v}_{\tau} \tilde{v} dy.
\end{equation}
By Hardy's inequality, $\sup_{\tau \in [1, T]} M(\tau) \lesssim \sup_{\tau \in [1, T]} E(\tilde{v})$. By direct computation,
\begin{equation}\label{5.2}
\aligned
\frac{d}{d\tau} M(t) = \int \frac{y}{(1 + |y|^{2}} \tilde{v}_{\tau \tau} \cdot \nabla \tilde{v} dy + \int \frac{y}{(1 + |y|^{2})^{1/2}} \tilde{v}_{\tau} \cdot \nabla \tilde{v}_{\tau} dy \\
+ \int \frac{1}{(1 + |y|^{2})^{1/2}} \tilde{v}_{\tau \tau} \tilde{v} dy + \int \frac{1}{(1 + |y|^{2})^{1/2}} \tilde{v}_{\tau}^{2} dy.
\endaligned
\end{equation}
Integrating by parts,
\begin{equation}\label{5.3}
\frac{1}{2} \int \frac{y}{(1 + |y|^{2})^{1/2}} \cdot \nabla(\tilde{v}_{\tau}^{2}) dy + \int \frac{1}{(1 + |y|^{2})^{1/2}} \tilde{v}_{\tau}^{2} dy = -\frac{1}{2} \int \frac{1}{(1 + |y|^{2})^{3/2}} \tilde{v}_{\tau}^{2} dy.
\end{equation}
Substituting $(\ref{4.30})$,
\begin{equation}\label{5.4}
\tilde{v}_{\tau \tau} = \Delta \tilde{v} - e^{-(p - 3) \tau} (\frac{|y|}{\sinh |y|})^{p - 1} |\tilde{v}|^{p - 1} \tilde{v} - e^{-(p - 3) \tau} (\frac{|y|}{\sinh |y|})^{p - 1} [|\tilde{u}|^{p - 1} \tilde{u} - |\tilde{v}|^{p - 1} \tilde{v} - |\tilde{w}|^{p - 1} \tilde{w}].
\end{equation}
Integrating by parts,
\begin{equation}\label{5.5}
\aligned
\int \frac{y}{(1 + |y|^{2})^{1/2}} \Delta \tilde{v} \cdot \nabla \tilde{v} dy + \int \frac{1}{(1 + |y|^{2})^{1/2}} \Delta \tilde{v} \tilde{v} dy \\
= -\frac{1}{2} \int \frac{1}{(1 + |y|^{2})^{1/2}} |\nabla \tilde{v}|^{2} - \frac{1}{2} \int \frac{|y|^{2}}{(1 + |y|^{2})^{3/2}} |\nabla \tilde{v}|^{2} + \int \frac{|y|^{2}}{(1 + |y|^{2})^{3/2}} |\partial_{r} \tilde{v}|^{2} - \frac{1}{2} \int \frac{1}{(1 + |y|^{2})^{5/2}} \tilde{v}^{2} dy \\
\leq -\frac{1}{2} \int \frac{1}{(1 + |y|^{2})^{3/2}} |\nabla \tilde{v}|^{2} - \frac{1}{2} \int \frac{1}{(1 + |y|^{2})^{5/2}} \tilde{v}^{2} dy.
\endaligned
\end{equation}

Next, integrating by parts,
\begin{equation}\label{5.6}
\aligned
-\int \frac{y}{(1 + |y|^{2})^{1/2}} e^{-(p - 3) \tau} (\frac{|y|}{\sinh |y|})^{p - 1} |\tilde{v}|^{p - 1} \tilde{v} \cdot \nabla \tilde{v} - \int \frac{1}{(1 + |y|^{2})^{1/2}} e^{-(p - 3) \tau} (\frac{|y|}{\sinh |y|})^{p - 1} |\tilde{v}|^{p + 1} dy \\
= -\frac{1}{p + 1} \int \frac{y}{(1 + |y|^{2})^{1/2}} e^{-(p - 3) \tau} (\frac{|y|}{\sinh |y|})^{p - 1} \cdot \nabla(|\tilde{v}|^{p + 1}) dy \\ - \int \frac{1}{(1 + |y|^{2})^{1/2}} e^{-(p - 3) \tau} (\frac{|y|}{\sinh |y|})^{p - 1} |\tilde{v}|^{p + 1} dy \\
= (\frac{3}{p + 1} - 1) \int \frac{1}{(1 + |y|^{2})^{1/2}} e^{-(p - 3) \tau} (\frac{|y|}{\sinh |y|})^{p - 1} |\tilde{v}|^{p + 1} dy \\ - \frac{1}{p + 1} \int \frac{|y|^{2}}{(1 + |y|^{2})^{3/2}} e^{-(p - 3) \tau} (\frac{|y|}{\sinh |y|})^{p - 1} |\tilde{v}|^{p + 1} \\ + \frac{(p - 1)}{p + 1} \int \frac{|y|}{(1 + |y|^{2})^{1/2}} e^{-(p - 3) \tau} (\frac{|y|}{\sinh |y|})^{p - 2} (\frac{\sinh |y| - |y| \cosh |y|}{\sinh^{2} |y|}) |\tilde{v}|^{p + 1} dy  \\
\leq -\frac{p - 2}{p + 1} \int \frac{1}{(1 + |y|^{2})^{1/2}} e^{-(p - 3) \tau} (\frac{|y|}{\sinh |y|})^{p - 1} |\tilde{v}|^{p + 1} dy.
\endaligned
\end{equation}

The error terms arising from
\begin{equation}\label{5.7}
e^{-(p - 3) \tau} (\frac{|y|}{\sinh |y|})^{p - 1} [|\tilde{u}|^{p - 1} \tilde{u} - |\tilde{v}|^{p - 1} \tilde{v} - |\tilde{w}|^{p - 1} \tilde{w}]
\end{equation}
can be handled similar to the error terms in the previous section. Recalling $(\ref{4.40})$, by $(\ref{4.50})$, $(\ref{4.51})$, and Hardy's inequality,
\begin{equation}\label{5.8}
\aligned
\int_{1}^{T} \int \frac{|y| |\nabla \tilde{v}| + |\tilde{v}|}{(1 + |y|^{2})^{1/2}} e^{-(p - 3) \tau} (\frac{|y|}{\sinh |y|})^{p - 1} |\tilde{w}|^{\frac{p - 1}{2}} |\tilde{v}|^{\frac{p + 1}{2}} dy d\tau \\ \lesssim \| e^{\tau} |w(e^{\tau} \cosh |y|, e^{\tau} \sinh |y| \frac{y}{|y|})|^{\frac{p - 1}{2}} \|_{L_{\tau}^{1} L_{y}^{\infty}} \| \nabla \tilde{v} \|_{L_{\tau}^{\infty} L_{y}^{2}} \cdot \sup_{\tau \in [1, T]} (\int e^{-(p - 3) \tau} (\frac{|y|}{\sinh |y|})^{p - 1} |\tilde{v}|^{p + 1} dy)^{1/2} \\
\lesssim \epsilon \sup_{\tau \in [1, T]} E(\tilde{v}(\tau)),
\endaligned
\end{equation}
and
\begin{equation}\label{5.9}
\aligned
\int_{1}^{T} \int \frac{|y| |\nabla \tilde{v}| + |\tilde{v}|}{(1 + |y|^{2})^{1/2}} e^{-(p - 3) \tau} (\frac{|y|}{\sinh |y|})^{p - 1} |\tilde{w}|^{p - 1} |\tilde{v}| dy d\tau \\ \lesssim \| e^{\tau} |w(e^{\tau} \cosh |y|, e^{\tau} \sinh |y| \frac{y}{|y|})|^{\frac{p - 1}{2}} \|_{L_{\tau}^{1} L_{y}^{\infty}} \| \nabla \tilde{v} \|_{L_{\tau}^{\infty} L_{y}^{2}}^{2} \| e^{\tau} |w(e^{\tau} \cosh |y|, e^{\tau} \sinh |y| \frac{y}{|y|})|^{\frac{p - 1}{2}} \|_{L_{\tau}^{\infty} L_{y}^{\infty}} \\
\lesssim \epsilon \sup_{\tau \in [1, T]} E(\tilde{v}(\tau)).
\endaligned
\end{equation}
Next, by H{\"o}lder's inequality,
\begin{equation}\label{5.10}
\aligned
\int_{1}^{T} \int \frac{1}{(1 + |y|^{2})^{1/2}} e^{-(p - 3) \tau} (\frac{|y|}{\sinh |y|})^{p - 1} |\tilde{v}|^{p} |\tilde{w}| dy d\tau \\
\lesssim (\int_{1}^{T} \int e^{-(p - 3) \tau} (\frac{|y|}{\sinh |y|})^{p - 1} |\tilde{v}|^{p + 1} dy d\tau)^{\frac{p - 2}{p - 1}} (\int_{1}^{T} \int \frac{1}{(1 + |y|^{2}} \frac{1}{\cosh^{2}(|y|)} |\tilde{v}|^{2} dy d\tau)^{\frac{1}{p - 1}} \\ \times \| e^{\frac{2}{p - 1} \tau} \cosh^{\frac{2}{p - 1}}(|y|) |w(e^{\tau} \cosh |y|, e^{\tau} \sinh |y| \frac{y}{|y|})| \|_{L_{\tau, y}^{\infty}} \\
\lesssim \epsilon \int_{1}^{T} \int e^{-(p - 3) \tau} (\frac{|y|}{\sinh |y|})^{p - 1} |\tilde{v}|^{p + 1} dy d\tau + \epsilon \int_{1}^{T} \int \frac{1}{(1 + |y|^{2})^{5/2}} |\tilde{v}|^{2} dy d\tau.
\endaligned
\end{equation}

Turning to
\begin{equation}\label{5.11}
\int \frac{y}{(1 + |y|^{2})^{1/2}} e^{-(p - 3) \tau} (\frac{|y|}{\sinh |y|})^{p - 1} \tilde{w} \cdot \nabla(|\tilde{v}|^{p - 1} \tilde{v}) dy,
\end{equation}
consider
\begin{equation}\label{5.12}
\int \frac{y}{(1 + |y|^{2})^{1/2}} e^{-(p - 3) \tau} (\frac{|y|}{\sinh |y|})^{p - 1} \tilde{w}_{j} \cdot \nabla(|\tilde{v}|^{p - 1} \tilde{v}) dy,
\end{equation}
for a fixed $j \in \mathbb{Z}$. Define a modified Littlewood--Paley function, this time in space. This function is similar to $(\ref{4.53})$. Let
\begin{equation}\label{5.13}
\tilde{P}_{0} f = \int \psi(y - z) f(z), \qquad \text{when} \qquad k > 0, \qquad \tilde{P}_{k} f = 2^{3k} \int \psi(2^{k}(y - z)) f(z) dz - 2^{3(k - 1)} \int \psi(2^{k - 1}(y - z)) f(z),
\end{equation}
where $\psi \in C_{0}^{\infty}(\mathbb{R}^{3})$ is supported on $|y| \leq 1$ and $\int \psi(y) dy = 1$. Now make a partition of unity
\begin{equation}\label{5.14}
1 = \sum_{m \geq 0} \chi(|y| - m).
\end{equation}
Define $k(m, \tau) = \sup \{ 0, \frac{m}{\ln(2)} + \frac{\tau}{\ln(2)} + j \}$. Since $\frac{|y|}{\sinh |y|} \sim \frac{|z|}{\sinh |z|}$ when $|y - z| \leq 1$,
\begin{equation}\label{5.15}
\aligned
\sum_{m \geq 0} \int \frac{y}{(1 + |y|^{2})^{1/2}} e^{-(p - 3) \tau} (\frac{|y|}{\sinh |y|})^{p - 1} \chi(|y| - m) \tilde{w}_{j}(\tau, y) \cdot \nabla (|P_{\leq k(m, \tau)}|^{p - 1} (P_{\leq k(m, \tau)} \tilde{v})) dy \\
\lesssim \sum_{m \geq 0} 2^{k(m, \tau) \cdot \frac{p - 3}{p - 1}} (\int_{m - 2 \leq |y| \leq m + 2} (\frac{|y|}{\sinh |y|})^{p - 1} e^{-(p - 3) \tau} |\tilde{v}|^{p + 1} dy)^{\frac{p - 2}{p - 1}} \\ \times (\int_{m - 2 \leq |y| \leq m + 2} \frac{1}{\cosh^{2} |y|} |\nabla \tilde{v}|^{2} + \frac{1}{|y|^{2} \cosh^{2} |y|} |\tilde{v}|^{2} dy)^{\frac{1}{p - 1}} \\ \times (\sup_{m - 2 \leq |y| \leq m + 2} |e^{\frac{2}{p - 1} \tau} \cosh^{\frac{2}{p - 1}} |y| w_{j}(e^{\tau} \cosh |y|, e^{\tau} \sinh |y| \frac{y}{|y|})|) \\
\lesssim a_{j} (\int (\frac{|y|}{\sinh |y|})^{p - 1} e^{-(p - 3) \tau} |\tilde{v}|^{p + 1} dy)^{\frac{p - 2}{p - 1}} (\int \frac{1}{\cosh^{2} |y|} |\nabla \tilde{v}|^{2} + \frac{1}{|y|^{2} \cosh^{2} |y|} |\tilde{v}|^{2} dy)^{\frac{1}{p - 1}}.
\endaligned
\end{equation}
Integrating by parts,
\begin{equation}\label{5.16}
\aligned
\sum_{m \geq 0} \int \frac{y}{(1 + |y|^{2})^{1/2}} e^{-(p - 3) \tau} (\frac{|y|}{\sinh |y|})^{p - 1} \chi(|y| - m) \tilde{w}_{j}(\tau, y) \cdot \nabla (|\tilde{v}|^{p - 1} |P_{\geq k(m, \tau)} \tilde{v}|) dy \\
= \sum_{m \geq 0} \int \frac{y}{(1 + |y|^{2})^{1/2}} e^{-(p - 3) \frac{p - 2}{p - 1} \tau} (\frac{|y|}{\sinh |y|})^{p - 2} \chi(|y| - m) e^{\frac{2}{p - 1} \tau} w_{j}(e^{\tau} \cosh |y|, e^{\tau} \sinh |y| \frac{y}{|y|}) \\ \cdot \nabla (|\tilde{v}|^{p - 1} |P_{\geq k(m, \tau)} \tilde{v}|) dy \\
= -\sum_{m \geq 0} \int \frac{y}{(1 + |y|^{2})^{1/2}} e^{-(p - 3) \frac{p - 2}{p - 1} \tau} (\frac{|y|}{\sinh |y|})^{p - 2} \chi(|y| - m)  (|\tilde{v}|^{p - 1} |P_{\geq k(m, \tau)} \tilde{v}|) \\ \cdot e^{\frac{2}{p - 1} \tau} [e^{\tau} \sinh |y| \frac{y}{|y|} (\partial_{t} w_{j}) + e^{\tau} \cosh |y| \frac{y}{|y|} (\partial_{r} w_{j})](e^{\tau} \cosh |y|, e^{\tau} \sinh |y| \frac{y}{|y|}) dy \\ 
-\sum_{m \geq 0} \int \frac{|y|}{(1 + |y|^{2})^{1/2}} e^{-(p - 3) \frac{p - 2}{p - 1} \tau} (\frac{|y|}{\sinh |y|})^{p - 2} \chi'(|y| - m)  e^{\frac{2}{p - 1} \tau} (w_{j})(e^{\tau} \cosh |y|, e^{\tau} \sinh |y| \frac{y}{|y|}) \\ \times (|\tilde{v}|^{p - 1} |P_{\geq k(m, \tau)} \tilde{v}|) dy \\ 
-\sum_{m \geq 0} \int \nabla \cdot (\frac{y}{(1 + |y|^{2})^{1/2}} e^{-(p - 3) \frac{p - 2}{p - 1} \tau} (\frac{|y|}{\sinh |y|})^{p - 2}) \chi(|y| - m)  e^{\frac{2}{p - 1} \tau} (w_{j})(e^{\tau} \cosh |y|, e^{\tau} \sinh |y| \frac{y}{|y|}) \\ \times (|\tilde{v}|^{p - 1} |P_{\geq k(m, \tau)} \tilde{v}|) dy.
\endaligned
\end{equation}
Then using the fundamental theorem of calculus, as in $(\ref{4.62})$, by Lemma $\ref{l4.1}$,
\begin{equation}\label{5.17}
\aligned
-\sum_{m \geq 0} \int \frac{y}{(1 + |y|^{2})^{1/2}} e^{-(p - 3) \frac{p - 2}{p - 1} \tau} (\frac{|y|}{\sinh |y|})^{p - 2} \chi(|y| - m)  (|\tilde{v}|^{p - 1} |P_{\geq k(m, \tau)} \tilde{v}|) \\ \cdot e^{\frac{2}{p - 1} \tau} [e^{\tau} \sinh |y| \frac{y}{|y|} (\partial_{t} w_{j}) + e^{\tau} \cosh |y| \frac{y}{|y|} (\partial_{r} w_{j})](e^{\tau} \cosh |y|, e^{\tau} \sinh |y| \frac{y}{|y|}) dy \\
\lesssim a_{j} (\int (\frac{|y|}{\sinh |y|})^{p - 1} e^{-(p - 3) \tau} |\tilde{v}|^{p + 1} dy)^{\frac{p - 2}{p - 1}} (\int \frac{1}{\cosh^{2} |y|} |\nabla \tilde{v}|^{2} + \frac{1}{|y|^{2} \cosh^{2} |y|} |\tilde{v}|^{2} dy)^{\frac{1}{p - 1}}.
\endaligned
\end{equation}
Also,
\begin{equation}\label{5.18}
\aligned
-\sum_{m \geq 0} \int \frac{|y|}{(1 + |y|^{2})^{1/2}} e^{-(p - 3) \frac{p - 2}{p - 1} \tau} (\frac{|y|}{\sinh |y|})^{p - 2} \chi'(|y| - m)  e^{\frac{2}{p - 1} \tau} (w_{j})(e^{\tau} \cosh |y|, e^{\tau} \sinh |y| \frac{y}{|y|}) \\ \times (|\tilde{v}|^{p - 1} |P_{\geq k(m, \tau)} \tilde{v}|) dy \\ 
-\sum_{m \geq 0} \int \nabla \cdot (\frac{y}{(1 + |y|^{2})^{1/2}} e^{-(p - 3) \frac{p - 2}{p - 1} \tau} (\frac{|y|}{\sinh |y|})^{p - 2}) \chi(|y| - m)  e^{\frac{2}{p - 1} \tau} (w_{j})(e^{\tau} \cosh |y|, e^{\tau} \sinh |y| \frac{y}{|y|}) \\ \times (|\tilde{v}|^{p - 1} |P_{\geq k(m, \tau)} \tilde{v}|) dy \\
\lesssim a_{j} (\int (\frac{|y|}{\sinh |y|})^{p - 1} e^{-(p - 3) \tau} |\tilde{v}|^{p + 1} dy)^{\frac{p - 2}{p - 1}} (\int \frac{1}{\cosh^{2} |y|} |\nabla \tilde{v}|^{2} + \frac{1}{|y|^{2} \cosh^{2} |y|} |\tilde{v}|^{2} dy)^{\frac{1}{p - 1}}.
\endaligned
\end{equation}
Therefore, we have proved
\begin{equation}\label{5.19}
\aligned
\int_{1}^{T} \int \frac{y}{(1 + |y|^{2})^{1/2}} e^{-(p - 3) \tau} (\frac{|y|}{\sinh |y|})^{p - 1} \tilde{w}_{j}(\tau, y) \cdot \nabla (|\tilde{v}|^{p - 1} \tilde{v})) dy d\tau \\ \lesssim \epsilon \int_{1}^{T} \int e^{-(p - 3) \tau} (\frac{|y|}{\sinh |y|})^{p - 1} |\tilde{v}|^{p + 1} dy d\tau + \epsilon \int_{1}^{T} \int \frac{1}{(1 + |y|^{2})^{5/2}} |\tilde{v}|^{2} dy d\tau,
\endaligned
\end{equation}
which completes the proof of Theorem $\ref{t4.5}$.

\section{Profile decomposition argument}
Having obtained a scattering result for any $u_{0} \in B_{1,1}^{\frac{3}{2} + s_{c}}$, $u_{1} \in B_{1,1}^{\frac{1}{2} + s_{c}}$, it only remains to show that this bound is uniform over all $(u_{0}, u_{1})$ satisfying
\begin{equation}\label{6.1}
\| (u_{0}, u_{1}) \|_{B_{1,1}^{\frac{3}{2} + s_{c}} \times B_{1,1}^{\frac{1}{2} + s_{c}}} \leq A,
\end{equation}
for some $A < \infty$. The proof argument is exactly parallel to the arguments in \cite{dodson2018global}, \cite{dodson2018global2}, and especially in \cite{dodson2018globalAPDE}. Here we are in the nonradial setting, however, we are aided by the fact that the nonlinearity is not the Lorentz invariant nonlinearity.

Let $(u_{0}^{n}, u_{1}^{n})$ be a bounded sequence in $B_{1,1}^{\frac{3}{2} + s_{c}} \times B_{1,1}^{\frac{1}{2} + s_{c}}$. Since this sequence is bounded in $\dot{H}^{s_{c}} \times \dot{H}^{s_{c} - 1}$, then by Theorem $3.1$ in \cite{ramos2012refinement}, we may make the profile decomposition
\begin{equation}\label{6.2}
S(t)(u_{0, n}, u_{1, n}) = \sum_{j = 1}^{N} \Gamma_{j}^{n} S(t)(\phi_{0}^{j}, \phi_{1}^{j}) + S(t)(R_{0, n}^{N}, R_{1, n}^{N}),
\end{equation}
where
\begin{equation}\label{6.3}
\lim_{N \rightarrow \infty} \limsup_{n \rightarrow \infty} \| S(t)(R_{0, n}^{N}, R_{1, n}^{N}) \|_{L_{t,x}^{2(p - 1)}(\mathbb{R} \times \mathbb{R}^{3})} = 0.
\end{equation}
The group $\Gamma_{j}^{n}$ is the group of operators generated by translation in space and in time, and also by the scaling symmetry. That is, there exist $x_{j}^{n} \in \mathbb{R}^{3}$, $t_{j}^{n} \in \mathbb{R}$, and $\lambda_{j}^{n} \in (0, \infty)$ such that
\begin{equation}\label{6.4}
\Gamma_{j}^{n} v(t,x) = (\lambda_{j}^{n})^{\frac{2}{p - 1}} v(\lambda_{j}^{n} (t - t_{j}^{n}), \lambda_{j}^{n} (x - x_{j}^{n})).
\end{equation}
Furthermore, the $\Gamma_{j}^{n}$'s have the asymptotic orthogonality property that when $j \neq k$,
\begin{equation}\label{6.5}
\lim_{n \rightarrow \infty} |\ln(\frac{\lambda_{j}^{n}}{\lambda_{k}^{n}})| + (\lambda_{j}^{n})^{1/2} (\lambda_{k}^{n})^{1/2} (|x_{j}^{n} - x_{k}^{n}| + |t_{j}^{n} - t_{k}^{n}|) = \infty.
\end{equation}
Using the dispersive estimate in $(\ref{1.9})$, $\frac{|t_{j}^{n}|}{\lambda_{j}^{n}}$ is uniformly bounded for any $j$.
\begin{lemma}\label{l6.1}
If $\frac{|t_{j}^{n}|}{\lambda_{j}^{n}} \rightarrow \infty$ then $\phi_{0}^{j} = 0$ and $\phi_{1}^{j} = 0$.
\end{lemma}
\begin{proof}
Indeed, from \cite{ramos2012refinement}, for any fixed $j$,
\begin{equation}\label{6.6}
\lim_{n \rightarrow \infty} (\Gamma_{j}^{n})^{-1} S(t)(u_{0}^{n}, u_{1}^{n}) \rightharpoonup S(t)(\phi_{0}^{j}, \phi_{1}^{j})
\end{equation}
weakly in $L_{t,x}^{2(p - 1)}$. Rewriting $(\Gamma_{j}^{n})^{-1}$,
\begin{equation}\label{6.7}
(\Gamma_{j}^{n})^{-1} S(t)(u_{0}^{n}, u_{1}^{n}) = S(t - \frac{t_{j}^{n}}{\lambda_{j}^{n}})((\lambda_{j}^{n})^{-\frac{2}{p - 1}} u_{0}^{n}(\frac{x - x_{j}^{n}}{\lambda_{j}^{n}}), (\lambda_{j}^{n})^{-\frac{p + 1}{p - 1}} u_{1}^{n}(\frac{x - x_{j}^{n}}{\lambda_{j}^{n}})),
\end{equation}
and then by the dispersive estimate $(\ref{1.9})$, for any fixed Littlewood--Paley projection, if $\frac{t_{n}^{j}}{\lambda_{n}^{j}} \rightarrow \pm \infty$,
\begin{equation}\label{6.8}
S(t - \frac{t_{j}^{n}}{\lambda_{j}^{n}})((\lambda_{j}^{n})^{-\frac{2}{p - 1}} u_{0}^{n}(\frac{x - x_{j}^{n}}{\lambda_{j}^{n}}), (\lambda_{j}^{n})^{-\frac{p + 1}{p - 1}} u_{1}^{n}(\frac{x - x_{j}^{n}}{\lambda_{j}^{n}})) \rightharpoonup 0,
\end{equation}
weakly in $L_{t,x}^{2(p - 1)}$, which proves the lemma.
\end{proof}

Since $t_{j}^{n}$ is bounded for any $j$, after passing to a subsequence, $t_{j}^{n} \rightarrow t_{j}$. Absorbing the remainder into $R_{N}$, we may rewrite $(\ref{6.2})$ with $\Gamma_{j}^{n}$ having no translation in time, that is,
\begin{equation}\label{6.9}
\Gamma_{j}^{n} v(t,x) = (\lambda_{j}^{n})^{\frac{2}{p - 1}} v(\lambda_{j}^{n} t, \lambda_{j}^{n} (x - x_{j}^{n})).
\end{equation}
Furthermore, since
\begin{equation}\label{6.10}
(\lambda_{j}^{n})^{\frac{2}{p - 1}} u_{0}(\lambda_{j}^{n} x) \rightharpoonup \phi_{0}^{j}, \qquad \text{and} \qquad (\lambda_{j}^{n})^{\frac{2}{p - 1} + 1} u_{1}(\lambda_{j}^{n} x) \rightharpoonup \phi_{1}^{j},
\end{equation}
we have the bounds
\begin{equation}\label{6.11}
\| \phi_{0}^{j} \|_{B_{1,1}^{\frac{3}{2} + s_{c}}} + \| \phi_{1}^{j} \|_{B_{1,1}^{\frac{1}{2} + s_{c}}} \leq A.
\end{equation}
Therefore, the solution to $(\ref{1.1})$ with initial data equal to $(\phi_{0}^{j}, \phi_{1}^{j})$ has a finite $L_{t,x}^{2(p - 1)}$ norm. Furthermore,
\begin{equation}\label{6.12}
\lim_{N \rightarrow \infty} \sum_{j = 1}^{N} \| (\phi_{0}^{j}, \phi_{1}^{j}) \|_{\dot{H}^{s_{c}} \times \dot{H}^{s_{c} - 1}}^{2} \leq \limsup_{n \rightarrow \infty} \| (u_{0,n}, u_{1, n}) \|_{\dot{H}^{s_{c}} \times \dot{H}^{s_{c} - 1}}^{2},
\end{equation}
so for only finitely many $j$, $\| (\phi_{0}^{j}, \phi_{1}^{j}) \|_{\dot{H}^{s_{c}} \times \dot{H}^{s_{c} - 1}} \geq \epsilon$. If $\| (\phi_{0}^{j}, \phi_{1}^{j}) \|_{\dot{H}^{s_{c}} \times \dot{H}^{s_{c} - 1}} \leq \epsilon$, then the solution to $(\ref{1.1})$ with initial data $(\phi_{0}^{j}, \phi_{1}^{j})$ has the bound
\begin{equation}\label{6.14}
\| u \|_{L_{t,x}^{2(p - 1)}} \lesssim \| (\phi_{0}^{j}, \phi_{1}^{j}) \|_{\dot{H}^{s_{c}} \times \dot{H}^{s_{c} - 1}}.
\end{equation}
Therefore, by standard perturbative arguments combined with the asymptotic orthogonality in $(\ref{6.5})$, if $u_{n}$ is the solution to $(\ref{1.1})$ with initial data $(u_{0,n}, u_{1, n})$,
\begin{equation}\label{6.13}
\lim_{n \rightarrow \infty} \| u_{n} \|_{L_{t,x}^{2(p - 1)}} < \infty.
\end{equation}
Thus, there must exist a uniform upper bound on the $L_{t,x}^{2(p - 1)}$ norm of a solution $u$ to $(\ref{1.1})$ whose initial data has bounded Besov norm.

\section*{Acknowledgements}
During the time of writing this paper, the author was partially supported by NSF Grant DMS-$1764358$. The author is also grateful to Andrew Lawrie and Walter Strauss for some helpful conversations at MIT on subcritical nonlinear wave equations.

\bibliography{biblio}
\bibliographystyle{alpha}

\end{document}